\documentclass[]{interact}

\usepackage{epstopdf}
\usepackage[caption=false]{subfig}

\theoremstyle{plain}
\newtheorem{theorem}{Theorem}[section]
\newtheorem{lemma}[theorem]{Lemma}

\newtheorem{proposition}[theorem]{Proposition}

\theoremstyle{definition}

\theoremstyle{remark}

\begin{document}


\title{Dynamic Decision Modeling for Viable Short and Long Term Production Policies: An HJB Approach}

\author{
\name{Achraf Bouhmady\textsuperscript{a,b}\thanks{CONTACT Achraf Bouhmady Email: achraf.bouhmady@um5r.ac.ma} and Mustapha Serhani \textsuperscript{c} and Nadia Raissi\textsuperscript{a}}
\affil{\textsuperscript{a} LAMA Laboratory, Faculty of Sciences, Mohammed V University in Rabat, Morocco; \\ \textsuperscript{b} LMI Laboratory, National Institute of Applied Sciences Rouen Normandie (INSA), Rouen, France; \\ \textsuperscript{c} TSAN Team, TSI Laboratory, FSJES, Moulay Ismail University of Meknes, Morocco}
}

\maketitle

\begin{abstract}
This study introduces a mathematical framework to investigate the viability and reachability of production systems under constraints. We develop a model that incorporates key decision variables, such as pricing policy, quality investment, and advertising, to analyze short-term tactical decisions and long-term strategic outcomes.
In the short term, we constructed a capture basin that defined the initial conditions under which production viability constraints were satisfied within the target zone. In the long term, we explore the dynamics of product quality and market demand to achieve and sustain the desired target. The Hamilton-Jacobi-Bellman (HJB) theory characterizes the capture basin and viability kernel using viscosity solutions of the HJB equation. This approach, which avoids controllability assumptions, is well suited to viability problems with specified targets. It provides managers with insights into maintaining production and inventory levels within viable ranges while considering product quality and evolving market demand.
We numerically studied the HJB equation to design and test computational methods that validate the theoretical insights. Simulations offer practical tools for decision-makers to address operational challenges while aligning with the long-term sustainability goals. This study enhances the production system performance and resilience by linking rigorous mathematics with actionable solutions.
\end{abstract}


\section{Introduction}

In the dynamic and competitive landscape of modern markets, companies are constantly looking for strategies to optimize pricing, quality, advertising, and inventory management to maintain their viability and profitability. This study aims to address these challenges by developing a comprehensive model that integrates dynamic pricing strategies, product quality management, advertising, and production decisions within the viability framework. The model leverages the Hamilton-Jacobi-Bellman (HJB) theory and viability theory, offering a novel approach to managing and optimizing market dynamics.

Previous research has extensively explored the interrelationships among pricing, advertising, product quality, and inventory management. Early foundational studies by Vidale and Wolfe (\cite{vidale1957operations}, \cite{yang2021optimal}) examined the effectiveness of advertising and its impact on sales dynamics, laying the groundwork for subsequent research that investigated the influence of pricing strategies on market penetration(\cite{robinson1975dynamic}, \cite{den2015dynamic}). These studies highlight the critical role of advertising in driving short-term market behavior, particularly in the diffusion of new products and technologies in the market.

However, the relationship between price and quality is nuanced. Traditional perspectives suggest that higher prices are typically associated with better quality, which can slow market diffusion due to premium pricing (\cite{tapiero1987reliability}, \cite{gavious2012price}). In contrast, more recent studies propose that lower prices might also indicate high quality, leading to faster market penetration \cite{heinsalu2021competitive}. These contrasting views underscore the importance of carefully balancing short- and long-term market strategies.

This study examines these complexities by integrating both short and long-term perspectives. The proposed model highlights the critical role of advertising as a key decision variable influencing pricing strategies, product quality investment, and market behavior. Its objective is to ensure that production and inventory levels remain within viable ranges despite market demand fluctuations driven by control variables. \\
 In contrast to classical adaptive control approaches, such as Model Predictive Control (MPC) \cite{kumar2012model}, which can face infeasibility issues under tight or conflicting constraints, our framework employs the Hamilton-Jacobi-Bellman (HJB) theory that characterizes capture basins and viability kernels using viscosity solutions of the HJB equation, avoids restrictive controllability assumptions, and provides global viability guarantees over extended time horizons. 
Our mathematical approach combines finite and infinite horizon HJB problems, building on recent work (\cite{rashkov2021model}, \cite{rashkov2022modeling}, \cite{altarovici2013general}, \cite{bokanowski2006anti}, \cite{assellaou2018value}) extending HJB theory to state-constrained control problems. This hybrid framework captures the complexities of short-term decisions and long-term strategic considerations, thereby aligning immediate actions with future goals. We demonstrate that the capture basin and viability kernel are definable through the level curves of a specific value function, $v$. This function is Lipschitz continuous and corresponds to a control problem without state constraints. Furthermore, $v$ is identified as the unique Lipschitz continuous viscosity solution to the Hamilton-Jacobi-Bellman (HJB) equation. 

This formulation enables the application of robust numerical schemes, including semi-Lagrangian and finite difference methods, to approximate the solution (\cite{assellaou2018value}, \cite{bokanowski2006anti}). Numerical simulations under various conditions were performed to investigate both short and long-term dynamics and validate the theoretical framework. 

These simulations underscore the robustness of the proposed control framework in addressing both short-term and long-term dynamic challenges.\\
By bridging the gap between short-term tactics and long-term strategies, this study provides a robust mathematical and practical framework. The short-term approach utilizes advertising as a control variable, characterized by the viscosity solution of the Hamilton-Jacobi-Bellman (HJB) equation, to maintain feasibility in production and inventory. In the long term, a hybrid finite-infinite horizon model ensures sustained competitiveness through strategic investments in quality. Supported by advanced numerical simulations, this dual-focus approach equips decision-makers with the tools to address dynamic challenges and achieve both immediate and enduring objectives.

The remainder of this paper is organized as follows. Section 2 formulates a comprehensive mathematical model that incorporates various features of production systems. Section 3 presents the general viability and reachability control problems. Section 4 analyzes the short-term viability of the stated problem using the Hamilton-Jacobi-Bellman (HJB) approach. Section 5 explores the long-term viability of the stated problem with the hybrid HJB approach, discussing the two phases of this long-term vision. Section 6 numerically illustrates the tendencies of viable control through direct optimization using the software package \textit{Roc-HJ Solver} \cite{RocHJ}, considering different features arising from viability analysis. Finally, the Appendix contains the numerical schemas used to solve the HJB equations adapted for both short-term and long-term viability of the problem. 
 
\section{The mathematical model}\label{model}
Our aim is to develop a dynamic production system that incorporates the key stages of pricing impact, advertising, and product quality. The proposed model focuses on optimizing production efficiency and addressing the durability aspect of the activity. It presents a more generalized approach to production and inventory management, considering both short-term and long-term objectives. Although the model is designed with a high level of generality without focusing on a specific industrial process, it significantly enhances our intellectual and conceptual understanding of production dynamics, durability, and decision-making in complex systems.

The model examines the decision-making process of a monopolist at each time t regarding product price, investment in product quality, and investment in advertising and production. These decisions play a crucial role in shaping market dynamics. Product price directly impacts consumer demand, influencing their perception and evaluation of the product. In contrast, investment in product quality directly affects the overall quality of the product, while advertising investments contribute to the diffusion and promotion of the product within the market. These decision variables collectively determine the monopolist's strategy and have significant implications for market outcomes and performance.\\
The analysis is conducted within the context of an infinite planning horizon, where $t\in \mathbb{R}^{+}$ is continuous.

\subsubsection{Inventory dynamics}
In this model, we consider a market where the inventory level at time $t$, denoted by $I(t)$, is subject to limitations imposed by the company's resource constraints, particularly its storage capacity. The inventory level is bounded within the interval $[I_{\text{min}}, I_{\text{max}}]$.

We introduced several variables to analyze this market. First, let $p(t) \geq 0$ represent the price charged at time $t$. Additionally, we denote $U(t) \geq 0$ as the advertising effort and $q(t)$ as the quality of the product. Furthermore, let $P(t)$ represent the production rate required to satisfy market demand. The demand for a product, considering the variables of price $p$, advertising effort $U$, and quality $q$, is denoted as $D(p,U,q)$.

The dynamics of inventory can be described by the following equation:

\begin{equation}\label{inventory}
	\dot{I}(t) = P(t) - D(p(t), U(t), q(t)) -\theta I(t), \quad I(0) = I_{0} \in [I_{min}, I_{max}]
\end{equation}
		
Here, $\dot{I}(t)$ represents the time derivative of the inventory level, and $I(0)$ represents the initial inventory level at $t=0$. The initial inventory level is assumed to be within the range $[I_{min}, I_{max}]$.
In contrast to prior research on dynamic pricing, which primarily focuses on the direct effects of price and quality on demand, our proposed model offers a unique perspective. Specifically, our model incorporates a negative relationship between demand and price ($\frac{\partial D}{\partial p} < 0$), indicating that an increase in the price leads to a decrease in demand. However, we also observe positive relationships between demand and both the reference price ($\frac{\partial D}{\partial r} \geq 0$) and quality ($\frac{\partial D}{\partial q} \geq 0$). This implies that an increase in the reference price or product quality corresponds to an increase in demand.
To illustrate these assumptions, we present a specific demand function that incorporates the aforementioned relationships. This parametric formulation takes the form of a linear demand function, given by

$$
D(p, U, q) = \rho U (a - b p + c q),
$$

In this equation, the parameters $\rho$, $a$, $b$, and $c$ are positive constants. Linear demand functions have been widely employed in academic research, as evidenced by several studies, such as those (\cite{hsieh2017optimal}, \cite{anton2023dynamic},\cite{zhang2020durable}).

This demand function is conceptually inspired by the models proposed by \cite{sethi2008optimal, anton2023dynamic} and \cite{chenavaz2020modeling}. However, it incorporates the simplifying assumptions of constant price sensitivity and a linear dependence on product quality.
Although we adopt this linear specification for clarity and analytical tractability, our qualitative findings are not confined to this functional form. The structure of our analysis remains valid under a broader class of demand functions, as long as they exhibit the following key properties: decreasing in price, non-decreasing in advertising effort, and non-decreasing with quality.\\
For instance, one could consider alternative formulations in which the effect of quality on demand follows a quadratic or convex pattern, capturing the increasing marginal returns to quality improvements. An example of such a specification is
$
D(p, U, q) = \rho U (a - b p^2 + c q),
$
This formulation preserves the core economic assumptions while allowing for more flexible representations of consumer behavior in response to quality changes. Crucially, our approach does not rely on the specific parametric form of the demand function, but rather on the general properties it satisfies.

\subsubsection{Quality dynamics}

The study of product quality has gained increasing attention in recent years, particularly in the context of firms' product innovation efforts. Researchers have invested considerable effort in exploring ways to enhance product quality, with some theoretical findings indicating a possible inverse relationship between price and quality (\cite{chenavaz2017better}, \cite{voros2019analysis}, \cite{chenavaz2021advertising}). A key contribution to this field was made by \cite{chenavaz2017analytical}, who identified specific conditions under which a price reduction can occur alongside an improvement in product quality based on an analysis of a general demand function.

Improving product quality involves various activities aimed at enhancing its attributes. \cite{chand1996capacity} proposed a framework that evaluates the intensity of these activities based on the time dedicated to quality improvements. Alternatively, studies by \cite{anton2023dynamic} and \cite{ni2019better} emphasize the intensity of quality enhancement efforts through monetary investments, denoted as $s(.)$. These researchers have developed dynamic models that effectively capture the evolving nature of product quality over time, providing frameworks to address the following dynamics:

\begin{equation}
	\frac{dq}{dt}(t) = K(s(t), q(t)), \quad q(0) = q_0,
\end{equation}

where $q(0)$ represents the initial quality of the product at $t=0$. The function $K: \mathbb{R}^{+} \times \mathbb{R}^{+} \rightarrow \mathbb{R}$ is assumed to be twice continuously differentiable and captures the relationship between investment in quality, expressed as $s(.)$, and the rate of change in product quality. Empirical evidence has consistently indicated a positive influence of quality investment on overall product quality. Consequently, an increase in the level of investment in quality is expected to enhance the overall quality of the product.

In our model, we adopt a specific parametric dynamic function based on the work of Ni et al. \cite{ni2019better}, given by

\[K(s, q) = \sqrt{s} - \delta q\]

The quality variable $q(.)$ is subject to certain constraints, as exceeding a certain limit would damage the company's reputation, while falling below a certain limit would result in the product being undesirable to the targeted consumers. Then $q(t) \in [q_{\text{min}}, q_{\text{max}}]$ 

\subsubsection{Production dynamics}

Let $P(t) \geq 0$ denote the production at time $t$, $I(t)$ the inventory level at time $t$, and $D(p,U,q)$ the market demand. The production dynamics at time $t$ depend on both the current inventory level, $I(t)$, and the current market demand, $D(p,U,q)$. 

However, the production process has specific limitations. First, there is a constraint on the productivity effort, considering the available resources such as equipment and workforce (labor). Consequently, the producer cannot surpass the maximum production level, denoted as $P_{max}$. This restriction ensures that the production process remains within the feasible bounds. Second, to optimize costs, it is imperative to establish a lower limit for production, denoted as $P_{min}$. By setting this limit, the objective is to avoid reducing production below a certain threshold, thus, maintaining a satisfactory level of output. Consequently, the problem at hand revolves around effectively managing production within the defined range of $[ P_{min}, P_{max} ]$.

In a general formulation, the dynamics of production can be expressed as

\begin{equation}\label{prod}
	\frac{dP}{dt}(t) = S(I(t), D(p(t),U(t),q(t))), \quad P(t_0) = P_{0},
\end{equation}

Here, $S: \mathbb{R}^{+} \times \mathbb{R}^{+} \rightarrow \mathbb{R}$ is a twice-continuously differentiable function, and $P_0$ represents the initial production level.

We assume that the function $S$ satisfies the following conditions for all $(I, D)$ in $\mathbb{R}^{+} \times \mathbb{R}^{+}$:
\begin{equation}\label{ass_prod}
	\frac{\partial S}{\partial D}>0, \quad \frac{\partial S}{\partial I}<0.
\end{equation}

These conditions indicate that a higher market demand leads to an increase in production activity ($\frac{\partial S}{\partial D}>0$), whereas a higher inventory level results in a decrease in production ($\frac{\partial S}{\partial I}<0$).

An illustrative example that conforms to assumption (\ref{ass_prod}) can be used to demonstrate the dynamics of the quality variable.
\[S(I,P,p,U,q)) = \alpha U \rho(a-b p+c q) - \beta I - \gamma P\]

We consider the parameters \(\beta\), and \(\gamma\) to be non-negative. Parameter \(\alpha\) represents the coefficient that quantifies the increase in production with respect to market demand. However, \(\beta\) denotes the coefficient that measures the decrease in production relative to the current inventory level. Finally, \(\gamma\) captures the impact of production decline over time, which can be attributed to factors such as reduced machine or labor productivity. Thus, \(\gamma\) accounts for the various factors that contribute to reduced productivity and serves as an indicator of this phenomenon.

Hence, throughout this study, we assume that \( \mathcal{H}_{1}: \alpha > \gamma \), which means that production is more responsive to shifts in market demand (\(\alpha\)) than to a gradual decline in productivity (\(\gamma\)). In other words, increases in demand have a more substantial positive impact on production levels than the negative impact of factors that reduce productivity, such as equipment wear or workforce fatigue.


\section{Viability analysis}
In this part, we shall focus on the inventory, Production and quality  model

	\begin{equation}
		\centering
		\left\{\begin{split}
		\frac{dP}{dt}(t)    &  = \alpha U(t) \rho(a-b p(t)+c q(t)) - \beta I(t)- \gamma P(t), \quad P(t_0) = P_{0}\\
			\frac{dq}{dt}(t) 	&  = \sqrt{s(t)} - \delta q(t), \quad q(t_0) = q_{0}\\
			\frac{dI}{dt}(t)         &  = P(t) - \rho U(t) (a-b p(t)+c q(t)) -\theta I(t), \quad  I(t_0) = I_{0}.			
		\end{split}\right.
	\end{equation}
	
where $a, b, c, \rho, \theta, \delta \in [0, +\infty)$ and the control vector $\omega$ is a measurable control function in $\mathcal{U}$ defined as
$$
\mathcal{U}:=\{\omega=(p, U, s ):[0,+\infty) \rightarrow[\underline{p}, \bar{p}]\times [0, \bar{U}]  \times [0, \bar{s}], \text { measurable}\}
$$

Understanding the importance of inventory and production management is critical to ensuring optimum operational efficiency while preserving an active workforce in the industry. We aim to meet the inventory management challenge while maintaining the desired production levels. We want to ensure that inventory and production remain within the viability constraint set, noted $$ K =\{ (P,q,I) \in \mathbb{R}^3_+  /\quad P \in [P_{\text{min}}, P_{\text{max}}]\; ; \; q \in [q_{\text{min}}, q_{\text{max}}]\; ; \; I\in [ I_{\text{min}}, I_{\text{max} }]\}$$
 
By establishing upper and lower limits for inventory and production levels, we aim to achieve an optimum balance between inventory availability and meeting production standards. It should be noted that production levels below the critical threshold ($P_{\text{min}}$) lead to higher unit costs, highlighting the importance of proactively controlling the dynamics of production and storage processes. This approach optimizes operational efficiency by fully exploiting all available resources while avoiding stock shortages or saturation.

Defining the sustainability framework defines the principles that guide inventory management and production decisions. This approach aims to optimize resource allocation and operational decisions, ensuring that stock levels remain acceptable while preserving production sustainability. The viability kernel represents the region of the initial state space where it is possible to find a control that maintains the system within the set of K constraints. By determining the viability kernel, we aim to establish a robust inventory management system that aligns with the organization's objectives.

But moreover in this set of constraints, we will have a target that changes according to the vision of viability 

\section{Short term viability}
First, we adopt a short-term perspective to identify a strategic combination that maintains production and inventories within viable limits. The duration of this short-term perspective can vary from one sector to another. For instance, in the IT sector, the sales of a specific new cell phone model represent a short-term view, considering that the average lifecycle of a mobile phone generation is 12 months before it is replaced by a new generation featuring improvements \cite{generation_iphone}. In this context, quality remains constant or comparable while exerting control through pricing and advertising strategies to manipulate market demand and ensure the viability of inventories and production levels. So, our model becomes

	\begin{equation}\label{equa2d}
	\centering
	\left\{\begin{split}
	\frac{dP}{dt}(t)    &  = \alpha U(t) \rho(a-b p(t)+c q) - \beta I(t)- \gamma P(t), \quad P(t_0) = P_0\\
		\frac{dI}{dt}(t)         &  = P(t) - \rho U(t) (a-b p(t)+c q) -\theta I(t), \quad  I(t_0) = I_{0}.		
	\end{split}\right.
\end{equation}

In this short-term vision, the manufacturer seeks to increase production while maintaining a maximum average stock level. This approach stems from the need always to have sufficient stock available to meet demand, especially when advertising anticipates instant demand. Studies such as that carried out by Wood et al. \cite{wood2009short} highlight the significant short-term effect of advertising on sales. So, by anticipating demand and adjusting production accordingly, manufacturers can optimize their results and seize market opportunities. Let \( \hat{K} \) denote the 2D restriction of set \( K \) at a fixed quality level \( q \). For a given \( q = \hat{q} \), \( \hat{K} \) takes the following form:

\[
\hat{K} = \{ (P, I) \in \mathbb{R}^2_+ \mid P \in [P_{\text{min}}, P_{\text{max}}]\; ; \; I \in [I_{\text{min}}, I_{\text{max}}] \}.
\]
We define the target set that models the set where the producer aims to end up in the short term at the final time by
$$C=\{ (P, I) \in \mathbb{R}^2_+ \mid P \in [\tilde{P}, P_{\text{max}}]\; ; \; I \in [\tilde{I}, I_{\text{max}}] \}, $$

where $\tilde{P} \in [P_{\text{min}}, P_{\text{max}}]$ and  $\tilde{I} \in [I_{\text{min}}, I_{\text{max}}],$ as illustrated in Fig \ref{Viab}, which shows the target set $C$ for the short-term vision in $\hat{K}$.
\begin{figure}[h!]
\centering
	\includegraphics[scale=0.3]{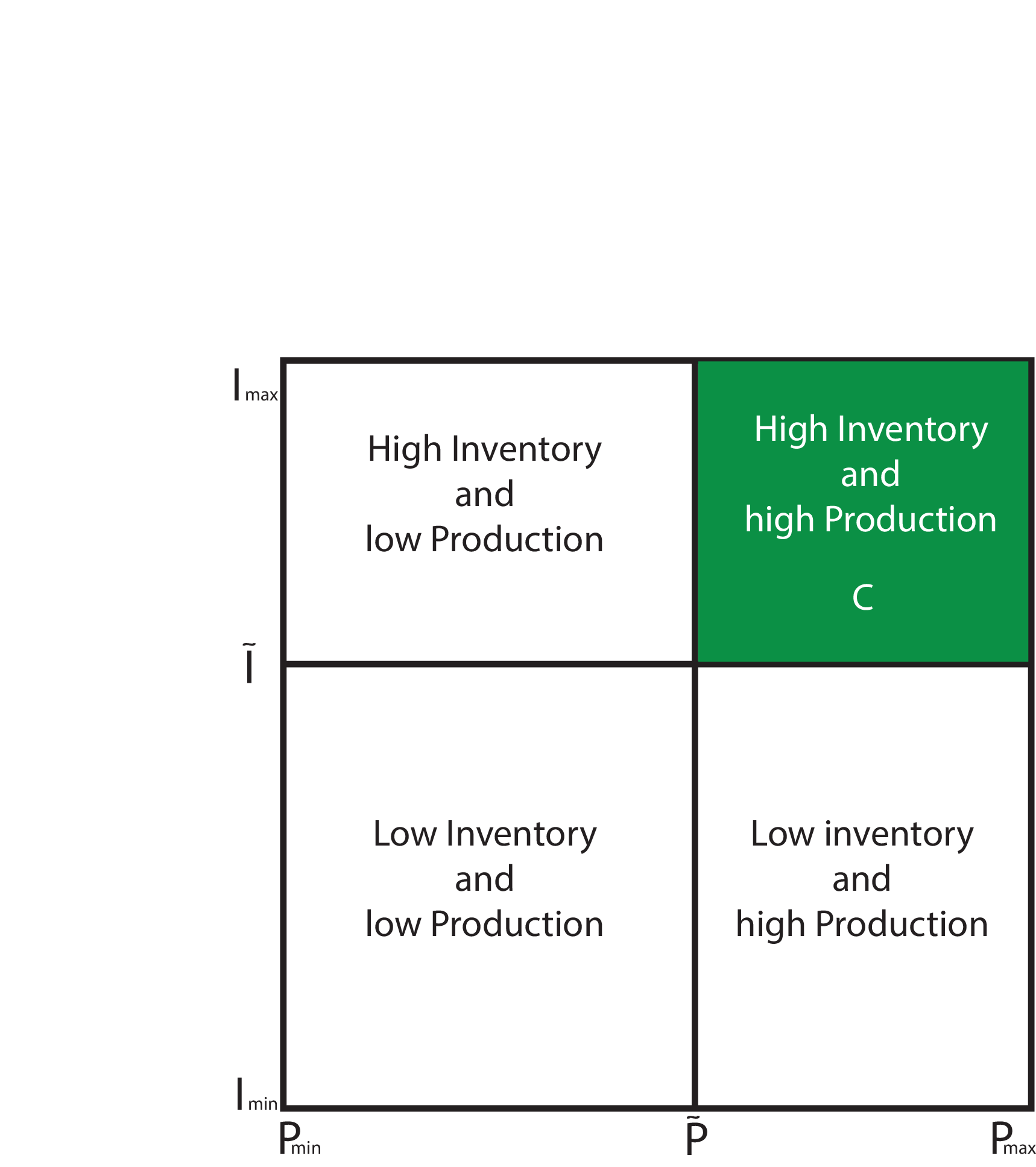}
	\caption{An illustration of the target C for the short-term vision in $\hat{K}$.}\label{Viab}
\end{figure}

Let $X=(P, I)$ represent the state vector, $\hat{K}$ and $C$ are non-empty closed sets, and the control variable $\omega=(p,U)$ (for simplicity, we conserve the same control variable notation as in the three-dimensional system). Our model formulates the problem in terms of viability, yielding the following system on $\Omega = \mathbb{R}^{2}_{+}$ for all $t\in [t_0, T]$,

\begin{subequations}\label{vidab}
	\begin{align}
		\dot{X}(t) &= f(X(t),\omega(t)) \label{viab_1}\\
		X(t) &\in \hat{K}, \quad X(T) \in C \quad \text{et} \quad \omega(\cdot) \in \hat{\mathcal{U}} \label{viab_2}\\
		X(0) &= (I_{0}, P_{0}) = y \in \hat{K} \label{viab_3}
	\end{align}
\end{subequations}

where, 
\begin{equation*}
	\begin{aligned}
		f(X,\omega) = & (\alpha U \rho(a-b p+c q) - \beta I- \gamma P , \\
		&  P - \rho U (a-b p+c q) -\theta I)
	\end{aligned}
\end{equation*}
for all $X\in \Omega$ and $\omega\in \mathbb{U}= [\underline{p}, \bar{p}]\times [0, \bar{U}]$. 
The control variable $\omega(\cdot)$ is a measurable control function taking values in the set $\hat{\mathcal{U}}$ defined as
$$
\hat{\mathcal{U}}:=\{\omega=(p, U ):[0,+\infty) \rightarrow \mathbb{U}, \text { measurable}\}
$$
In the remainder of this article, we denote the solution of equation (\ref{vidab}), controlled by \(\omega\) and starting at $\mathbf{y}$ as \(\mathbf{X}_{\mathbf{y}}^{\omega}\).

\begin{lemma}
	The function $f(\cdot,\omega)$ is Lipschitz continuous on $\Omega$, for all fixed $\omega\in  \mathbb{U}$ 
\end{lemma}

\begin{proof}
The function \( f \) is given, for all $(P,I,\omega)\in \Omega\times \mathbb{U}$, by

\[
f(P, I,\omega) = \begin{pmatrix}
\alpha U \rho (a - b p + c q) - \beta I - \gamma P\\
P - \rho U (a - b p + c q) - \theta I 
\end{pmatrix}.
\]

To demonstrate that \( f \) is Lipschitz continuous, we rewrite \( f \) in vector form as

\[
f(x, \omega) = A X + f_0(\omega),
\]

where \( X = \begin{pmatrix} P \\ I \end{pmatrix} \), $\omega\in \mathbb{U}$ and \( A \) is the matrix:

\[
A = \begin{pmatrix}
-\beta & -\gamma \\
1 & -\theta
\end{pmatrix}.
\]

The Lipschitz constant \( L \) is given by the norm of matrix \( A \).

\end{proof}

\begin{lemma}
Given $\omega \in \mathbb{U} $, under $\mathcal{H}_{1}$, system (\ref{vidab}) admits a unique locally stable equilibrium $(P^{*},I^{*})$.

$$
P^{*} =  \frac{\rho U (a - b p + c q) (\beta + \theta \alpha)}{\beta + \gamma \theta}, \quad I^{*}=\frac{(\alpha - \gamma) ( \alpha U \rho(a-b p+c q))}{\gamma \theta +\beta}.
$$
\end{lemma}

\begin{proof}
	 Given $ \omega \in [\underline{p}, \bar{p}]\times [0, \bar{U}] $,  under $\mathcal{H}_{1}$ the dynamic system (\ref{vidab}) has a unique equilibrium $(P^{*},I^{*})$, the local stability of $(P^{*},I^{*}) \in \mathbb{R}^{2}_{+}$ is determined by the sign of the eigenvalues of the Jacobian matrix, we have 

	 \begin{equation*}
	 	\mathbb{J} = 
	 	\begin{pmatrix}
	 		-\gamma & -\beta  \\
	 		1 & -\theta 
	 	\end{pmatrix}.
	 \end{equation*}
 $ \det (\mathbb{J}) = \theta \gamma + \beta > 0 $ and $ \text{tr} (\mathbb{J}) < 0 $. This implies that the eigenvalues have the same sign and a negative real part. Therefore, $(P^{*},I^{*})$ is locally asymptotically stable.
\end{proof}

As we saw in the lemma,  system (6) admits a locally stable equilibrium. However, there is always a risk of violating the constraint \( \hat{K} \) over a certain time interval, even if the equilibrium $(P^*, I^*)$ is within constraint \( \hat{K} \). To address this, we propose using the level set approach to characterize the capture basin (backward reachable set), which ensures that the system remains viable within the constraint \( \hat{K} \) and reaches a specified target at a final time \( T \). 

\subsection{Variational analysis for the capture basin}\label{short_viab}

Consider the problem in (\ref{vidab}), which involves characterizing, for every \(t \in [0, T]\), the set of all initial positions from which it is possible to find an admissible trajectory that reaches the target at time \(T\) while remaining within the set on \([t, T]\):

\[ 
\operatorname{Cap}_{c}(t) := \left\{ y \in \Omega: \exists\, \omega \in \hat{\mathcal{U}} \text{ s.t. } \mathbf{X}_{\mathbf{y}}^\omega(T) \in \mathcal{C} \text{ and } \mathbf{X}_{\mathbf{y}}^\omega(s) \in \hat{K} \text{ for } s \in (t, T) \right\}.
\]

In the classical approach, the viability kernel can be determined by studying the Lyapunov function (for instance, see (\cite{bayen2019minimal})). However, when the goal is to determine the capture basin $\operatorname{Cap}_{c}(t)$ (viability kernel with a target), it becomes more complex to do so using the Lyapunov function.

To address this, we transformed the problem into a Hamilton-Jacobi-Bellman (HJB) equation. This involves finding a cost functional that is minimized to obtain the associated value function denoted \( v \). This value function is a solution to the HJB equation. We adopt the approach described by \cite{altarovici2013general}, who solved an alternative cost minimization problem that included an exact penalty for the state constraint.

First, we consider two continuous Lipschitz functions $v_{0}$ and $g$ such that

\[ x \in \hat{K} \Leftrightarrow g(x) \leq 0 \]
\[ x \in C \Leftrightarrow v_{0}(x) \leq 0 \]

The choice of \( v_0 \) and \( g \) may vary depending on the problem (see, for  example, \cite{altarovici2013general}, \cite{assellaou2018value}).

We consider the following problem:

\begin{equation}\label{utility}
v(t, \mathbf{y}) = \inf_{\omega \in \hat{\mathcal{U}}}\left\{ \max \left(  v_0(\mathbf{X}_{\mathbf{y}}^\omega(T)), \sup_{\theta \in (t, T)} g\left(\mathbf{X}_{\mathbf{y}}^\omega(\theta)\right) \right)  \right\},
\end{equation}

The problem in (\ref{utility}) does not explicitly constrain the state. Instead, the term $\sup_{\theta \in (t, T)} g\left(\mathbf{X}_{\mathbf{y}}^\omega(\theta)\right)$ serves as a penalty incurred by the trajectory \( \mathbf{X}_{\mathbf{y}}^\omega \) if the state constraints are violated. In Theorem \ref{main} we demonstrate that the benefit of addressing problem (\ref{utility}) is that the value function \( v \) can be uniquely determined as the continuous viscosity solution of a Hamilton-Jacobi-Bellman (HJB) equation.

Recall that the distance function to a set $\mathcal{M}$ is given by 
\[
\mathrm{d}_{\mathcal{M}}(\mathbf{x}) = \inf\{ |x - y| : y \in \mathcal{M} \}.
\]

We define \( g \) and $v_{0}$ as follows

\begin{equation}\label{g1}
g(\mathbf{x}) = 
\begin{cases} 
\mathrm{d}_{\hat{K}}(\mathbf{x}), & \mathbf{x} \in \Omega \setminus \hat{K}, \\ 
-\mathrm{d}_{\Omega \setminus \hat{K}}(\mathbf{x}), & \mathbf{x} \in \hat{K} \cap \Omega.
\end{cases}
\end{equation}

Similarly, 

\begin{equation}\label{v01}
v_{0}(\mathbf{x}) = 
\begin{cases} 
\mathrm{d}_{C}(\mathbf{x}), & \mathbf{x} \in \Omega \setminus C, \\ 
-\mathrm{d}_{\Omega \setminus C}(\mathbf{x}), & \mathbf{x} \in C \cap \Omega.

\end{cases}
\end{equation}
Where $\Omega\setminus \mathcal{M}:=\{y\in \Omega \mbox{ and } y\notin \mathcal{M} \}$.

\begin{theorem}\label{thm1}
 Let \( v \) be the value function defined in (\ref{utility}). Then, for every \( t \geq 0 \), the capture basin is given by
\[ 
\operatorname{Cap}_{c}(t) = \{\mathbf{y} \in \Omega \mid v(t, \mathbf{y}) \leq 0\}.
\]
\end{theorem}

\begin{proof}
Assume \( v(t, \mathbf{y}) \leq 0 \). Given that the dynamic \( f(\cdot,\omega) \) is Lipschitz continuous for fixed $\omega\in \mathbb{U}$, and \( f(x, \mathbb{U}) \) is convex for each $x$, an admissible trajectory \( \mathbf{X}_{\mathbf{y}}^\omega \) exists such that
\[
\max \left( v_0(\mathbf{X}_{\mathbf{y}}^\omega(T)), \sup_{\theta \in (t, T)} g\left(\mathbf{X}_{\mathbf{y}}^\omega(\theta)\right) \right) \leq 0.
\]

Therefore, for all \( \theta \in (t, T) \),
\[
g\left(\mathbf{X}_{\mathbf{y}}^\omega(\theta)\right) \leq 0,
\]
This implies that \( \mathbf{X}_{\mathbf{y}}^\omega(\theta) \in \hat{K} \) for all \( \theta \in (t, T) \).

Additionally,
\[
v_0(\mathbf{X}_{\mathbf{y}}^\omega(T)) \leq 0,
\]
This implies that \( \mathbf{X}_{\mathbf{y}}^\omega(T) \in C \). Thus, this defines \( \operatorname{Cap}_{C}(t) \).\\
To prove the other inclusion, assume \( \mathbf{y} \in \operatorname{Cap}_{C}(t) \). By the definition of the capture basin \( \operatorname{Cap}_{C}(t) \), there exists an admissible control \( \omega \in \hat{\mathcal{U}} \) and an associated trajectory \( \mathbf{X}_{\mathbf{y}}^\omega(s) \) for \( s \in [t, T] \) such that

For all \( s \in (t, T) \),
\[
\mathbf{X}_{\mathbf{y}}^\omega(s) \in \hat{K},
\]
which implies that
\[
g\left(\mathbf{X}_{\mathbf{y}}^\omega(s)\right) \leq 0.
\]

At the terminal time \( T \),
\[
\mathbf{X}_{\mathbf{y}}^\omega(T) \in C,
\]
which implies that
\[
v_0\left(\mathbf{X}_{\mathbf{y}}^\omega(T)\right) \leq 0.
\]

Now, consider the value function \( v(t, \mathbf{y}) \), defined by
\[
v(t, \mathbf{y}) = \inf_{\omega \in \hat{\mathcal{U}}} \max \left( v_0\left(\mathbf{X}_{\mathbf{y}}^\omega(T)\right), \sup_{s \in (t, T)} g\left(\mathbf{X}_{\mathbf{y}}^\omega(s)\right) \right).
\]
Since we have an admissible control \( \omega \) such that
\[
\max \left( v_0\left(\mathbf{X}_{\mathbf{y}}^\omega(T)\right), \sup_{s \in (t, T)} g\left(\mathbf{X}_{\mathbf{y}}^\omega(s)\right) \right) \leq 0,
\]
it follows that
\[
v(t, \mathbf{y}) \leq 0.
\]

This establishes the other inclusion, so we conclude that
\[
\operatorname{Cap}_{C}(t) = \left\{ \mathbf{y} \in \Omega \mid v(t, \mathbf{y}) \leq 0 \right\},
\]
This completes the proof of the theorem.
\end{proof}

Now, we consider the following lemma, which will be used to prove the Lipschitz continuity of $v$.

\begin{lemma}\label{3}
For $h \geq 0$, the function $v$ may be defined according to the following property: for all $t \in [0,T]$, $s \geq 0$, and $y \in \Omega$,

\begin{equation}\label{DPP}
v(t, y)=\inf \left\{\max \left(v\left(t+s, x_y^\alpha(t+s)\right), \sup_{\theta \in(t, t+s)} g\left(x_y^\alpha(\theta)\right)\right), \alpha \in  \hat{\mathcal{U}} \right\} ;
\end{equation}
\end{lemma}

\begin{proof}
The proof of the Dynamic Programming Principle in Eq.(\ref{DPP}) utilizes standard methods and can be demonstrated through the same reasoning found in sources (\cite{barron1999viscosity}, \cite{barron1990semicontinuous}).
\end{proof}

\begin{proposition}
	For every $T>0,  v$ is Lipschitz continuous on $[0, T] \times \Omega $.
\end{proposition}

\begin{proof}
Let $T > 0$. For $y, y' \in \Omega$ and $t \in [0, T]$, using the definition of $\vartheta$, we have

	$$
	\begin{aligned}
		 \lvert v(t, y)-v(t, y^{'}) \rvert &= \left| \inf_{\omega \in \hat{\mathcal{U}}} \max \left( v_0(\mathbf{X}_{y}^{\omega}(T)), \sup_{\theta \in (t, T)} g(\mathbf{X}_{y}^{\omega}(\theta)) \right) - \inf_{\omega \in \hat{\mathcal{U}}} \max \left( v_0(\mathbf{X}_{y'}^{\omega}(T)), \sup_{\theta \in (t, T)} g(\mathbf{X}_{y'}^{\omega}(\theta)) \right) \right| \\
	&=	\left|  \inf _{\omega \in \hat{\mathcal{U}}}   \max \left\{   v_ {0}(\mathbf{\mathbf{X}_{\mathbf{y}}^\omega(T)}) - v_ {0}(\mathbf{\mathbf{X}_{\mathbf{y^{'}}}^\omega(T)}) , \sup _{\theta \in(t,T)}  g\left(\mathbf{X}_{\mathbf{y}}^\omega(\theta) \right) - g\left(\mathbf{X}_{\mathbf{y^{'}}}^\omega(\theta)\right)  \right\} \right| \\
		&\leq \sup _{\omega \in \hat{\mathcal{U}}}   \max \left\{ \lvert  v_ {0}(\mathbf{\mathbf{X}_{\mathbf{y}}^\omega(T)}) - v_ {0}(\mathbf{\mathbf{X}_{\mathbf{y^{'}}}^\omega(T)}) \rvert, \sup _{\theta \in(t,T)} \lvert  g\left(\mathbf{X}_{\mathbf{y}}^\omega(\theta) \right) - g\left(\mathbf{X}_{\mathbf{y^{'}}}^\omega(\theta)\right) \rvert \right\} \\
		& \leq \sup _{\omega \in \hat{\mathcal{U}}}\left(L_C\lvert \mathbf{\mathbf{X}_{\mathbf{y}}^\omega(T)} - \mathbf{\mathbf{X}_{\mathbf{y^{'}}}^\omega(T)} \rvert, L_g \sup _{\theta \in(t,T)}\lvert\mathbf{X}_{\mathbf{y}}^\omega(\theta)  - \mathbf{X}_{\mathbf{y^{'}}}^\omega(\theta) \rvert\right),
	\end{aligned}
	$$
	where $L_C$ and $L_g$ denote the Lipschitz constants of $v_0$ and $g$, respectively. Due to the Lipschitz continuity of function $f(\cdot,\omega)$,  by the Gronwall Lemma that  $ \forall \theta \in [0,T]$ $ \lvert \mathbf{X}_{\mathbf{y}}^\omega(\theta) - \mathbf{X}_{\mathbf{y^{'}}}^\omega(\theta) \rvert \leq e^{L \theta} \lvert y-y^{\prime} \lvert \leq e^{L T} \lvert y-y^{\prime} \lvert $, where $L=max(L_g ,L_c)$. Therefore we conclude that
	
\begin{equation}\label{lep}
	\lvert v(t, y)-v(t, y^{\prime}) \rvert \leq \max \left(L_C, L_g \right) e^{L T} \lvert y-y^{\prime} \lvert 
\end{equation}

Let $y \in \Omega$ and $s \geq 0$. From Lemma \ref{3}, we can deduce that $v(t, y) \geq g(y)$.
Therefore,
$$
\begin{aligned}
	\lvert v(t+s, y)-v(t, y) \rvert  &= \lvert \inf_{\omega\in \hat{\mathcal{U}}}\max \left(v\left(t,\mathbf{X}_y^\omega(t+s)\right), \sup_{\theta \in(t,t+s)} g\left(\mathbf{X}_y^\omega(\theta)\right)\right)- \max \left( v(t, y), g(y)  \right)  \rvert  \\
	& \leq \sup_{\omega\in \hat{\mathcal{U}}}\max \left( \lvert  v\left(t,\mathbf{X}_y^\omega(t+s)\right) - v(t, y) \rvert ,    \sup_{\theta \in(t,t+s)}\lvert  g\left(\mathbf{X}_y^\omega(\theta)\right) -g(y)  \rvert  \right)  \\
	&  \leq \sup_{\omega\in \hat{\mathcal{U}}} \max \left( \max \left(L_C, L_g \right) e^{L T} \lvert \mathbf{X}_y^\omega(t+s)-y   \lvert , L_g \sup_{\theta \in(t,t+s)} \lvert\mathbf{X}_y^\omega(\theta) -y  \rvert  \right)  \\
\end{aligned}
$$

where we have used (\ref{lep}).\\

Furthermore, let us denote $ C_f = \max _{a \in \mathbb{U}} \lvert f(0, a) \lvert<\infty$.

we have $ \lvert f(y, a)\rvert \leq C_f+L \lvert y \lvert$. Consequently, by applying Grönwall's lemma, we can deduce that $$\lvert \mathbf{X}_y^\omega(\theta)-y\rvert \leq\left(C_f+L \lvert y \lvert\right) \mathrm{e}^{L s} s \leq \left(C_f+L \lvert y \lvert\right) \mathrm{e}^{L T} s$$ for $\theta \in (t, t+s)$.

Hence,
\[\lvert v(t+s, y)-v(t, y) \rvert \leq M s \]

where $ M> 0$. In addition, the following proposition holds.
\end{proof}

\begin{theorem}\label{main}
	
The function value $v$ is the unique continuous viscosity solution of the Hamilton-Jacobi-Bellman equation 

\begin{equation} \label{HJB}
\begin{cases}
    \min \left\{ -\partial_{t} v(t,y) + H(y, \nabla v), \, v(t, y) - g(y) \right\} = 0, & \text{for } y \in \Omega, \\
    v(T,y) = \max(v_{0}(y), g(y)),
\end{cases}
\end{equation}

and
\begin{align*}
	H(\mathbf{X}, \zeta) & = \max_{\omega \in \mathbb{U}} - \left\langle f(\mathbf{X},\omega),\zeta \right\rangle \\
	& =  (\theta I - P)    \zeta_2 + (\beta I + \gamma P )  \zeta_1  + max \left\{ (\zeta_2 - \alpha \zeta_1) (\rho \bar{U} (a-b \underline{p} +cq)), 0 \right\},               
\end{align*}
where $  \mathbf{X}=(P, I) $, $\zeta= (\zeta_1, \zeta_2)$ and $(\bar{U},\underline{p})$ is the point where the maximum is achieved.

\end{theorem}

\begin{proof}
First, we demonstrate that \( v \) is a solution of (\ref{HJB}).

\textit{i) Verification that \( v \) is a Supersolution}

Let $y \in \mathbb{R}^2$ and $t > 0$, and let $\phi$ be a function in \( C^1 \) such that \((t, y)\) is a local minimum of \( v - \phi \). Therefore, we can assume \( v(t, y) = \phi(t, y) \). According to Lemma $\ref{3}$, we have that for any $ 0 \leq s << 1 $
\[
\inf_{\omega\in \hat{\mathcal{U}}} \left( v(t+s, \mathbf{X}_y^{\omega}(t+s))\right)  \leq v(t, y),
\]
and since
\[
\phi(t+s, \mathbf{X}_y^{\omega}(t+s)) \leq v(t+s, \mathbf{X}_y^{\omega}(t+s)),
\]
It follows that
\[
\phi(t+s, \mathbf{X}_y^{\omega}(t+s)) \leq \phi(t, y).
\]
Using a classical argument, we obtain
\[
-\partial_{t} \phi(t,y) + H(t, y, \nabla \phi) \geq 0.
\]

Moreover, by the definition of \( v \), for every \((t, y) \in [0, T] \times \Omega\), we have
\[
v(t, y) \geq \inf_{\omega\in \hat{\mathcal{U}}
} \sup_{\theta \in (t, T)} g(\mathbf{X}_y^\omega(\theta)) \geq g(y).
\]
Combining these two inequalities, we get
\[
\min \{\partial_{t} \phi (t,y) + H(y, \nabla \phi), v(t, y) - g(y)\} \geq 0,
\]
This implies that \( v \) is a supersolution of (\ref{HJB}).

\textit{ii) Verification that \( v \) is a Subsolution}

Now, we show that \( v \) is a subsolution. Let \( y \in \mathbb{R}^2 \) and \( t > 0 \), and let \( \phi \) be a function in \( C^1 \) such that \((t, y)\) is a local maximum of \( v - \phi \).\\
From Lemma \ref{3} (DPP), we have 
$v(t+s, \mathbf{X}_y^{\omega}(t+s)) = \displaystyle\inf_{\omega\in \hat{\mathcal{U}}} v(t, y), \quad \text{for any } 0 \leq s \ll 1,
$
and due to the continuity of \( v \), we have
$
\phi(t+s, \mathbf{X}_y^{\omega}(t+s)) \geq \displaystyle\inf_{\omega \in \hat{\mathcal{U}}} \phi(t, y).
$
Using a classical argument, we obtain
\[
\min \{-\partial_{t} \phi(t,y) + H(y, \nabla \phi), v(t, y) - g(y)\} \leq 0,
\]
This proves that \( v \) is a subsolution. Hence, \( v \) is a viscosity solution of (\ref{HJB}).

The uniqueness of the solution \( v \) to equation (\(\ref{HJB}\)) is established using the classical comparison principle (see, for instance, \cite{barles1994solutions}, \cite{altarovici2013general}). Specifically, the Hamiltonian \( H(\mathbf{X}, \zeta) \) is defined as
\[
H(\mathbf{X}, \zeta) = (\theta I - P)    \zeta_2 + (\beta I + \gamma P )  \zeta_1  + max \left\{ (\zeta_2 - \alpha \zeta_1) (\rho \bar{U} (a-b \underline{p} +cq)), 0 \right\},
\]

where $  \mathbf{X}=(P, I) $ and $\zeta= (\zeta_1, \zeta_2)$. The Hamiltonian \( H(\mathbf{X}, \zeta) \) satisfies the following properties:

\[
\begin{aligned}
\left| H\left(\mathbf{X}_2, \zeta\right) - H\left(\mathbf{X}_1, \zeta\right) \right| &\leq C_1(1 + |\zeta|) \left| \mathbf{X}_2 - \mathbf{X}_1 \right|, \\
\left| H\left(\mathbf{X}, \zeta_2\right) - H\left(\mathbf{X}, \zeta_1\right) \right| &\leq C_1 \left| \zeta_2 - \zeta_1 \right|,
\end{aligned}
\]

for some constant \( C_1 \geq 0 \) and for all \( \mathbf{X}_i, \zeta_i, \mathbf{X}, \) and \(\zeta\) in \(\mathbb{R}^2\), where $i=1,2$.

Because the Hamiltonian \( H \) satisfies these inequalities, the comparison principle ensures the uniqueness of the solution \( v \). In other words, the bounds on the differences in \( H \) with respect to \( \mathbf{X} \) and \(\zeta\) imply that any two solutions must coincide, thereby establishing the uniqueness of \( v \).

\end{proof}

While these strategic adjustments are critical for short-term performance, they set the stage for addressing long-term viability issues. As companies optimize their immediate strategies, they must also consider how these decisions influence their ability to scale and sustain their operations over time. Specifically, the focus must shift toward balancing increased production with efficient inventory management to ensure the long-term success of the industry.



\section{Long term viability}

In the long term, the manufacturer's vision is to raise production with minimal average inventory, as the latest is a heavy resource mobilizer (cash, warehousing, conditioning, transport, etc.). Companies need help reducing inventory-related KPIs, such as Inventory turnover rate and stock-to-sales ratio. These indicators should be lowered. This means that the main challenge would be to balance higher production with lower inventory time.\\

In the long-term vision, production is a state variable that is affected by product quantity and quality. Likewise, inventory is a state variable that sales, pricing, and production can influence, whereas investments impact product quality. Additionally, we assume that the demand satisfied by each company is equivalent to its sales, as advertising has a limited impact on long-term demand because factors such as product quality and competition can vary. Therefore, in the long term, companies should prioritize managing their production, pricing, and product quality to reflect the ongoing nature of these challenges, and the dynamic model becomes for $t \in [0, + \infty[$.

\begin{equation}
	\centering
	\left\{\begin{split}
	\frac{dP}{dt}(t)    &  = \alpha U \rho(a-b p(t)+c q(t)) - \beta I(t)- \gamma P(t), \quad P(0) = P_0\\
		\frac{dq}{dt}(t) 	&  = \sqrt{s(t)} - \delta q(t), \quad q(0) = q_0\\
		\frac{dI}{dt}(t)         &  = P(t) - \rho U (a-b p(t)+c q(t)) -\theta I(t), \quad  I(0) = I_{0}.
	\end{split}\right.
\end{equation}

Therefore, the producer's main long-term objective is to strike a balance between producing a large number of products and maintaining a low level of inventory. The new question is whether there is a control that can bring the system back in line with this objective. \\ Mathematically, let \(\Omega = \mathbb{R}_+^3\), and $X=(P, q, I)$ represent the state vector. Our model formulates a long-term (with $T= +\infty$) viability problem, resulting in the following system of equations.

\begin{equation}
\left\{
\begin{aligned}
    &\dot{X}(t) = f(X(t), \omega(t)), \quad t \in (0, +\infty), \\
    & X(0) = \tilde{y} = (I_{0}, q_{0}, P_{0}) \in K, \quad \omega(\cdot) \in \mathcal{U}, \\
    & X(t) \in K, \quad t \in (0, +\infty), \\
    &X(t) \in C', \quad t \in [t_{1}, +\infty),    
\end{aligned}
\right.
\end{equation}

where, 

\begin{equation*}
	\begin{aligned}
		f(X,\omega) = & ( \alpha U \rho(a-b p+c q - \beta I- \gamma P, \\
		& \sqrt{s} - \delta q, \\
		& P - \rho U (a-b p+c q -\theta I )
	\end{aligned}
\end{equation*}
for all $X\in \Omega$, $w\in \mathbb{V}= [\underline{p}, \bar{p}]\times [0, \bar{v}]$. The control state $\omega(\cdot)$ is a measurable control function taking values in the new set $\mathcal{U'}$ defined as
$$
\mathcal{U'}:=\{\omega=(p, s ):[0,+\infty) \rightarrow \mathbb{V}, \text { measurable}\}
$$

Additionally, the constraint $K$ is defined as
\[
K = \{ (I,q,P) \in \Omega \mid P \in [P_{\text{min}}, P_{\text{max}}], \quad q \in [q_{\text{min}}, q_{\text{max}}], \quad I \in [I_{\text{min}}, I_{\text{max}}]\}.
\]
and the target set that models the set where the producer aims to end up in the long term  is given by
$$C'=\{ (P, q, I) \in \mathbb{R}^2_+ \mid P \in [P', P_{\text{max}}]\; ; \;    q \in [q',\tilde{q}] \; ; \;I \in [I', I_{\text{max}}] \}, $$

where $P' \in [P_{\text{min}}, P_{\text{max}}]$, $q', \tilde{q}\in [q_{\text{min}}, q_{\text{max}}]$ and $I'\in [I_{\text{min}}, I_{\text{max}}].$\\
The parameters defining $C'$ reflect the producer's strategic choices for sustainable operations. The production threshold $P'$ establishes the minimum required output level, while the quality bounds $q'$ and $\tilde{q}$ determine the minimum acceptable quality standard and the ideal quality target, respectively. The inventory floor $I'$ ensures adequate stock levels to maintain the operations. These targets are carefully designed to respect the system's hard constraints ($C' \subset K$), enabling the producer to maintain both production viability and quality standards in the long term.

Finally, the time $t_{1} $ represents the entry time into target zone $C^{'}$.

\subsection{Variational analysis for the long term viability}\label{long_viab}

In the first step, we search for the viability kernel \(\operatorname{Viab}_{C'}\) for the long-term target zone \(C'\). This viability kernel \(\operatorname{Viab}_{C'}\) is  the set of initial states ensuring that the system remains viable in \(C'\) under the control actions, defined by:
\[
\operatorname{Viab}_{C'} := \left\{ \mathbf{y}  \in \Omega : \exists\, \omega(\cdot) \in \mathcal{U'} \text{ such that } \mathbf{X}_{\mathbf{y}}^\omega(s) \in C', \text{ for } s \in (t_{1}, +\infty) \right\},
\]

Moreover, we will target  \(\operatorname{Viab}_{C'}\) at \(t = t_{1}\) to guarantee that the system remains viable in \(C'\). For \(t_{1} > 0\), we define the capture basin for the new target zone \(\operatorname{Viab}_{C'}\) as:

\[
\operatorname{Cap}_{\operatorname{Viab}_{C'}}(t) := \left\{ \mathbf{y}  \in \Omega : \exists\, \omega(\cdot) \in \mathcal{U'} \text{ s.t. } \mathbf{X}_{\mathbf{y}}^\omega(t_{1}) \in \operatorname{Viab}_{C'} \text{ and } \mathbf{X}_{\mathbf{y}}^\omega(s) \in K \text{ for } s \in (t, t_{1}) \right\}.
\]

Therefore, using the same argument as in Section \ref{short_viab}, we introduce the long-term value function (penalized problem) defined for $\mathbf{y}  = X(t_1)\in \Omega$ and $\lambda > L $, as

\begin{equation}\label{uti_viab}
\tilde{v}(\mathbf{y} )= \inf\left\{ \sup_{\theta \in (t_{1}, +\infty)} e^{-\lambda \theta} \tilde{v}_0 \left(\mathbf{X}_{\mathbf{y}}^\omega(\theta)\right), \omega \in \mathcal{U'} \right\},
\end{equation}

where,
\begin{equation}\label{v0l}
\tilde{v}_{0}(\mathbf{x}) = \begin{cases} 
\mathrm{d}_{C'}(\mathbf{x}), & \mathbf{x} \in \Omega \setminus C', \\ 
-\mathrm{d}_{\Omega \setminus C'}(\mathbf{x}), & \mathbf{x} \in C' \cap \Omega.

\end{cases}
\end{equation}

\begin{theorem}\label{HJBV_long}
For every \( t \geq t_{1} \),
\begin{itemize}
\item[i.]  The viability kernel is characterized by
\[ 
\operatorname{Viab}_{C'} = \{\mathbf{y} \in \Omega \mid \tilde{v}( \mathbf{y}) \leq 0\}.
\]

\item[ii.] The function $\tilde{v}$ is the unique Lipchitz continuous viscosity solution of the following EDP

\begin{equation}\label{HJB1}
\min \left\{ \lambda \tilde{v}(\mathbf{y}) + H(\mathbf{y}, \nabla \tilde{v}), \, \tilde{v}(\mathbf{y}) - \tilde{v}_{0}(\mathbf{y}) \right\} = 0, \quad  y \in \Omega, 
\end{equation}
where the Hamiltonian \( H \) is given by

\begin{align*}
  H(y, \zeta) & = (\theta I - P) \zeta_2 + (\beta I + \gamma P) \zeta_1 \\
  &+ \max \left\{ \left( \zeta_2 - \alpha \zeta_1 \right) (\rho \bar{U} (a - b \underline{p} + cq)), 0 \right\} + \zeta_3 \left( \sqrt{\bar{s}} - \delta q \right),
\end{align*}
\end{itemize}
Where, $y=(P, q, I)$ and $\zeta= (\zeta_1, \zeta_2, \zeta_3 )$.
\end{theorem}

\begin{proof}
 To prove (i), we use the same argument as that in Theorem \ref{thm1}.
 
 (ii) For all control $ \omega : (t_{1}, +\infty) \rightarrow \mathbb{V}$, we denote $\omega_{1}$ the restriction of $\omega$ on $(t_{1},t_{1}+ h)$ and $\omega_{2}(t)= \omega(t+h)$, we have for all $h \geq 0$, $\mathbf{y} \in \Omega$
 
\begin{equation*}
\sup_{\theta \in (t_{1}, +\infty)} e^{- \lambda \theta}\tilde{v}_{0}(X_{\mathbf{y}}^{\omega}(\theta)) = \max\lbrace \sup_{\theta \in (t_{1}, +\infty)} e^{- \lambda (\theta + h )}\tilde{v}_{0}(X^{\omega_2}_{X^{\omega_1}_\mathbf{y(t_{1}+h)}(\theta)} ),\sup_{\theta \in (t_{1}, t_{1}+ h)} e^{- \lambda \theta}\tilde{v}_{0}(X_{\mathbf{y}}^{\omega_{1}}(\theta))\rbrace
\end{equation*}

Then: 
\begin{equation}\label{DDP_2}
\tilde{v}(\mathbf{y}) = \inf \left\lbrace \max\lbrace  e^{- \lambda  h }\tilde{v}(X_{\mathbf{y}}^{\omega}(t_{1}+ h) ),\sup_{\theta \in (t_{1}, t_{1}+ h)} e^{- \lambda \theta}\tilde{v}_{0}(X_{\mathbf{y}}^{\omega}(\theta))\rbrace , \quad \omega : (t_{1}, t_{1} +h) \rightarrow \mathbb{V}  \right\rbrace 
\end{equation}

Let \( \phi \in C^1 \) such that  $ \mathbf{y}$ is a local minimum of \( \tilde v - \phi \). Therefore, we can assume \( \tilde v (\mathbf{y}) = \phi(\mathbf{y}) \), and we have for $h \geq 0$ small enough

\begin{equation}\label{min_loc}
\phi(\mathbf{X}_\mathbf{y}^{\omega}(t_{1}+ h) ) \leq \tilde v(\mathbf{X}_\mathbf{y}^{\omega}((t_{1}+ h) ),
\end{equation}

According to equations (\ref{DDP_2}) and (\ref{min_loc}),
\begin{equation}
\inf_{\omega \in \mathcal{U'}}\varphi(X_{\mathbf{y}}^{\omega}(t_{1}+ h) ) \leq e^{\lambda  h } \varphi(\mathbf{y}),
\end{equation}

as $X^{\omega}_{\mathbf{y}}(t_{1}+h)=\mathbf{y} +h f(x,\omega)+o(h)$, so
$$
\varphi\left(X^{\omega}_{\mathbf{y}}(t_{1}+h)\right)=\varphi(\mathbf{y})+\nabla \varphi(\mathbf{y}) f(x,\omega)+o(h),
$$
and 

$$
e^{ \lambda h} = 1+ \lambda h+ o(h),
$$

Therefore, we have for $h \rightarrow 0$

\[ \min \left\{ \lambda \varphi(\mathbf{y}) + H(\mathbf{y}, \nabla \varphi), \, \tilde{v}(\mathbf{y}) - \tilde{v}_{0}(\mathbf{y}) \right\} \geq 0\]
in the viscosity sense. so we $\tilde{v}(\mathbf{y}) $  is a supersolution.

Now we prove that $\tilde{v}(\mathbf{y}) $ is a subsolution, let  \( \phi \in C^1 \) such that  $ \mathbf{y}$ is a local maximum of \( \tilde v - \phi \). According to (\ref{DDP_2}) we have 

\begin{equation}
\varphi(\mathbf{y})(X_{\mathbf{y}}^{\omega}(t_{1}+ h) ) \geq \tilde{v}(X_{\mathbf{y}}^{\omega}(t_{1}+ h) ) \geq e^{-\lambda  h } \tilde{v}(\mathbf{y})= \varphi(\mathbf{y}),
\end{equation}

By classical arguments, we have that $\lambda \varphi(\mathbf{y}) + H(\mathbf{y}, \nabla \varphi) \leq 0$, and by the definition of $\tilde{v}(\mathbf{y}) $ we have that: 

\[   \tilde{v}(\mathbf{y}) = \inf\left\{ \sup_{\theta \in (t_{1}, +\infty)} e^{-\lambda \theta} \tilde{v}_0 \left(\mathbf{X}_{\mathbf{y}}^\omega(\theta)\right), \omega \in \mathcal{U'} \right\} \leq  \tilde{v}_{0}(\mathbf{y})\]

Therefore, 
\[ \min \left\{ \lambda \varphi(\mathbf{y}) + H(\mathbf{y}, \nabla \varphi), \, \tilde{v}(\mathbf{y}) - \tilde{v}_{0}(\mathbf{y}) \right\} \leq 0\] in the viscosity sense. This is the desired outcome.

\end{proof}

After determining \(\operatorname{Viab}_{C'}\), the target at \(t = t_{1}\), the objective is to find \(\operatorname{Cap}_{\operatorname{Viab}_{C'}}\), which leads to the target \(\operatorname{Viab}_{C'}\) and  Consider

\begin{equation}\label{g}
g(\mathbf{x}) = 
\begin{cases} 
\mathrm{d}_{K}(\mathbf{x}), & \mathbf{x} \in \Omega \setminus K, \\ 
-\mathrm{d}_{\Omega \setminus K}(\mathbf{x}), & \mathbf{x} \in K \cap \Omega.
\end{cases}
\end{equation}

Similarly, 

\begin{equation}\label{v0}
v_{0}(\mathbf{x}) = 
\begin{cases} 
\mathrm{d}_{\operatorname{Viab}_{C'}}(\mathbf{x}), & \mathbf{x} \in \Omega \setminus \operatorname{Viab}_{C'}, \\ 
-\mathrm{d}_{\Omega \setminus \operatorname{Viab}_{C'}}(\mathbf{x}), & \mathbf{x} \in \operatorname{Viab}_{C'} \cap \Omega.

\end{cases}
\end{equation}

Recall that the value function $v$ is given by \ref{utility}, as 

\begin{equation*}
v(t, \mathbf{y}) = \inf_{\omega \in \mathcal{U'}}\left\{ \max \left(  v_0(\mathbf{X}_{\mathbf{y}}^\omega(t_1)), \sup_{\theta \in (t, t_1)} g\left(\mathbf{X}_{\mathbf{y}}^\omega(\theta)\right) \right)  \right\},
\end{equation*}

According to theorem (\ref{main}) we have that 

\[ 
\operatorname{Cap}_{\operatorname{Viab}_{C'}}(t) = \{y  \in \Omega \mid v(t, y) \leq 0\},
\]
such that $v$ is the unique viscosity solution of

\begin{equation} \label{HJB2}
\begin{cases}
    \min \left\{ -\partial_{t} v(t,y) + H(y, \nabla v), \, v(t, y) - g(y) \right\} = 0, & \text{for } y \in \Omega, \\
    v(t_{1},y) = \max(v_{0}(y), g(y)),
\end{cases}
\end{equation}

where the Hamiltonian \( H \) is given by:

\begin{align*}
  H(y, \zeta) & = (\theta I - P) \zeta_2 + (\beta I + \gamma P) \zeta_1 \\
  &+ \max \left\{ \left( \zeta_2 - \alpha \zeta_1 \right) (\rho \bar{U} (a - b \underline{p} + cq)), 0 \right\} + \zeta_3 \left( \sqrt{\bar{s}} - \delta q \right).  
\end{align*}

Therefore, we established a framework to determine the viability kernel \(\operatorname{Viab}_{C'}\) and capture basin \(\operatorname{Cap}_{\operatorname{Viab}_{C'}}\) for the target zone \(C'\). By defining the long-term value function \(\tilde{v}(y)\) and proving its properties using viscosity solutions, we showed that the viability kernel and capture basin can be characterized using specific Hamilton-Jacobi equations. This provides a robust mathematical foundation for ensuring system viability within the target zone under various control actions in the future.\\
Furthermore, by determining \(\operatorname{Cap}_{\operatorname{Viab}_{C'}}\), we ensure that a control strategy exists to keep the system state within \(C'\) during the time interval \((t_1, \infty)\). This guarantees the long-term viability of the system within the desired target zone \(C'\).

Next, we transition to the numerical simulation phase, where we implement the theoretical results and demonstrate their practical application. We will visualize the short-term and long-term capture basin by solving the Hamilton-Jacobi equations numerically.


\section{Numerical analysis}
We employ a numerical partial differential equation (PDE) approach to address the viability and reachability problems given in Sect. \ref{short_viab} and \ref{long_viab}, following the schemas detailed in the Appendix. This approach was implemented using the ROC-HJ package, a C++-based software parallelized with OpenMP, with the parameters specified in Table \ref{parameters1}. MATLAB was used to reconstruct the controls and corresponding viability trajectories. The selected parameters were based on relevant studies, including \cite{sethi2008optimal}, \cite{anton2023dynamic}, and \cite{kaszkurewicz2018modeling}, to provide a representative set of results and demonstrate general numerical behavior without specific optimization.


\begin{table}[htbp]
\centering
\caption{Parameter values of the production system model.}
\label{parameters1}
\begin{tabular}{p{0.4\linewidth}c}
\toprule
\textbf{Parameter} & \textbf{Value} \\
\midrule
\(\rho\) & \(0.2\) \\
\(a\) & \(20\) \\
\(b\) & \(0.2\) \\
\(c\) & \(0.3\) \\
\(\theta\) & \(0.1\) \\
\(\delta\) & \(0.2\) \\
\(\alpha\) & \(0.8\) \\
\(\beta\) & \(0.1\) \\
\(\gamma\) & \(0.1\) \\
\(\bar{U}\) & \(1\) \\
\(\bar{p}\) & \(150\) \\
\(\underline{p}\) & \(75\) \\
\(r\) & \(0.04\) \\
\bottomrule
\end{tabular}
\end{table}

We begin by focusing on the short-term dynamics of the production system, where the goal is for the system to reach the target set $C$ within time horizon $T$. This is addressed by solving the corresponding Hamilton-Jacobi equations to ensure that the system remains within the operational limits while meeting the short-term objectives. These simulations assessed the system's ability to make rapid adjustments, offering key insights into immediate control strategies and performance.


\subsection{Short term numerical simulations}\label{shortnumeric}
In the short-term numerical analysis, we aim to compute the capture basin for the problem described in Equation (\ref{vidab}). We assumed a quality factor of \( q = 40\) during this phase, which simplified the problem to a two-dimensional space. Our primary focus is on controlling the dynamics of inventory \( I \) and production \( P \), ensuring that the system remains within the predefined capacity limits while pursuing a specific short-term objective.

To achieve our objectives, we employed the Lax-Friedrichs (LF) scheme for the numerical approximation of the partial differential equation (PDE) model \eqref{HJB}, utilizing the parameters listed in Table \ref{parameters1}. This discretization method provides a robust framework for effectively approximating the HJB equation, which is crucial for the short-term control of the system. 
The specific configuration parameters used in {\tt Roc-hj} for this example are listed in Table \ref{Tab.1}.

\begin{table}[htbp]
\centering
\caption{Discretization scheme and {\tt Roc-hj} settings.}
\label{Tab.1}
\begin{tabular}{p{0.4\linewidth}c}
\toprule
\textbf{Description} & \textbf{Value} \\
\midrule
Discretization method & {\tt Lax-Friedrichs} \\
The constraint \(\hat{K}\) & \([1, 4] \times [1, 4]\) \\
The target \(C\) & \([3, 4] \times [3, 4]\) \\
The time horizon & \([0, 1]\) \\
\bottomrule
\end{tabular}
\end{table}

In this context, the system state \( X(t) = (P(t), I(t)) \) evolves in a 2D space. Our objective is to guide the system towards the target set \( C \) within the time horizon while adhering to the system constraint.
The system dynamics are governed by Eq. \eqref{Viab}. Consequently, we refer to the situation outlined in the previous sections, where the associated Hamilton-Jacobi-Bellman (HJB) equation is

\begin{equation*} 
\begin{cases}
    \min \left\{ -\partial_{t} v(t,y) + H(y, \nabla v), \, v(t, y) - g(y) \right\} = 0, & \text{for } y \in \Omega, \\
    v(T,y) = \max(v_{0}(y), g(y)),
\end{cases}
\end{equation*}

where the Hamiltonian \( H \) is given by:

\begin{align*}
    H(X, \zeta) & = (\theta I - P) \zeta_2 + (\beta I + \gamma P) \zeta_1 + \max \left\{ \left( \zeta_2 - \alpha \zeta_1 \right) (\rho \bar{U} (a - b \underline{p} + cq)), 0 \right\}.
\end{align*}

In this case, the terminal condition is given by

\[
v(T, y) = \max(v_{0}(y), g(y)),
\]

where the auxiliary function \( v_0(y) \) is defined as

\[
v_0(y) = \min \left\{ r_C, \| y - a \|_\infty - r_C \right\},
\]

with the parameters:

\[
r_C = \frac{I_{\text{max}} - \tilde{I}}{2} = 0.5, \quad a = \left(\frac{I_{\text{max}} + \tilde{I}}{2}, \frac{P_{\text{max}} + \tilde{P}}{2}\right) = (3.5, 3.5).
\]

Additionally, the function \( g(y) \) is defined as

\[
g(y) = \min \left\{ r_0, \| y - a_0 \|_\infty - r_0 \right\},
\]

with the parameters:

\[
r_0 = 2.5, \quad a_0 = (2.5, 2.5).
\]

The numerical scheme is described in detail in Appendix \ref{app}. We use a discretization with $1025^2$ spatial mesh points over the area $\hat{K}=[1, 4] \times [1, 4]$, and $T = 1$, representing one complete cycle of the product's lifecycle. The CFL condition of this case is $CFL=0.8 < 1$, which guarantees the convergence of the numerical scheme.

\begin{figure}[h!]
	\centering
	\includegraphics[width=0.8\textwidth]{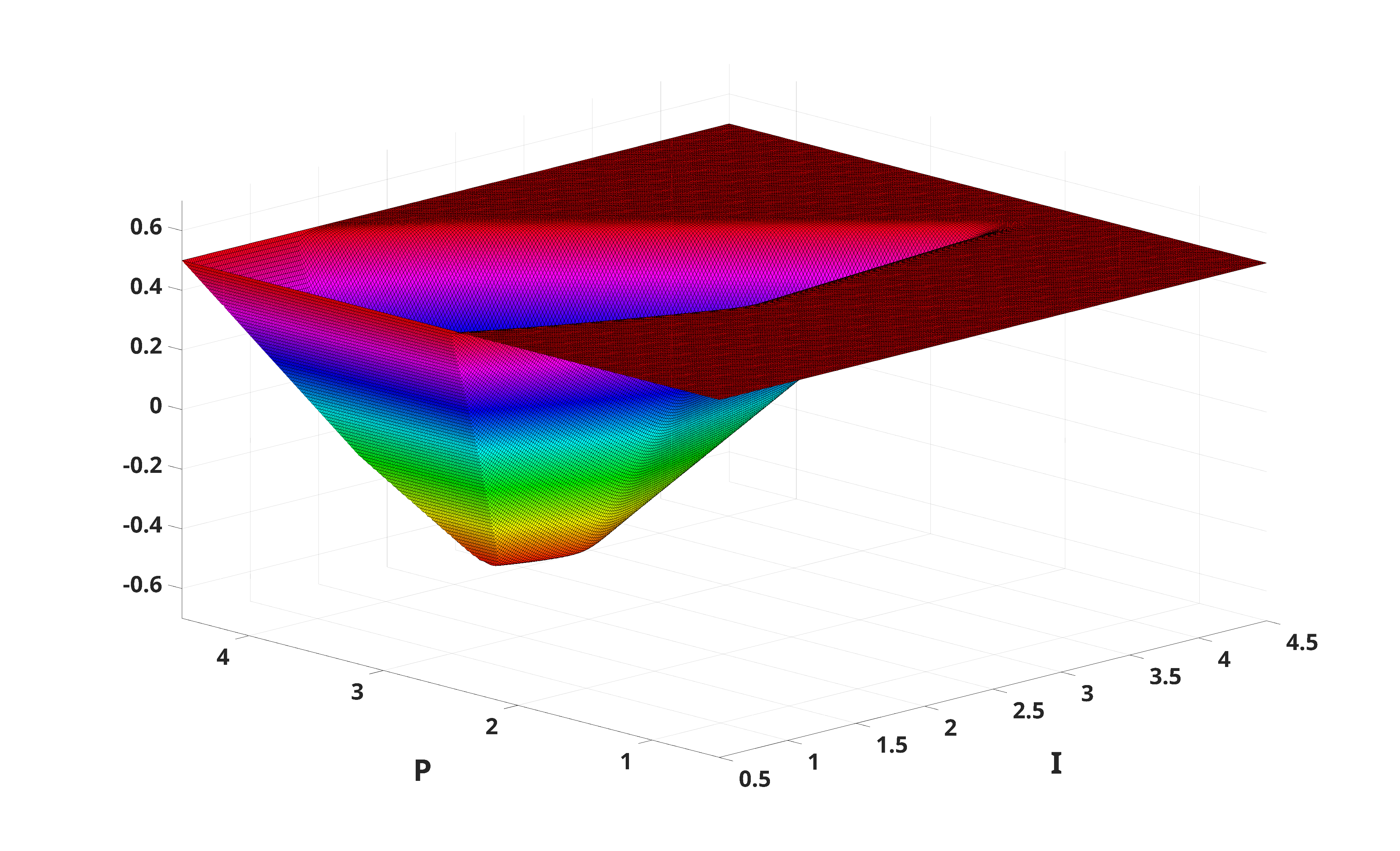}	
	\caption{Capture basin for Viable Production and Inventory Management: Variation of function value $v$. The plot shows the variation of function values.}
	\label{capt_graph_3d}
\end{figure}

\begin{figure}[h!]
	\centering
	\includegraphics[width=0.6\textwidth]{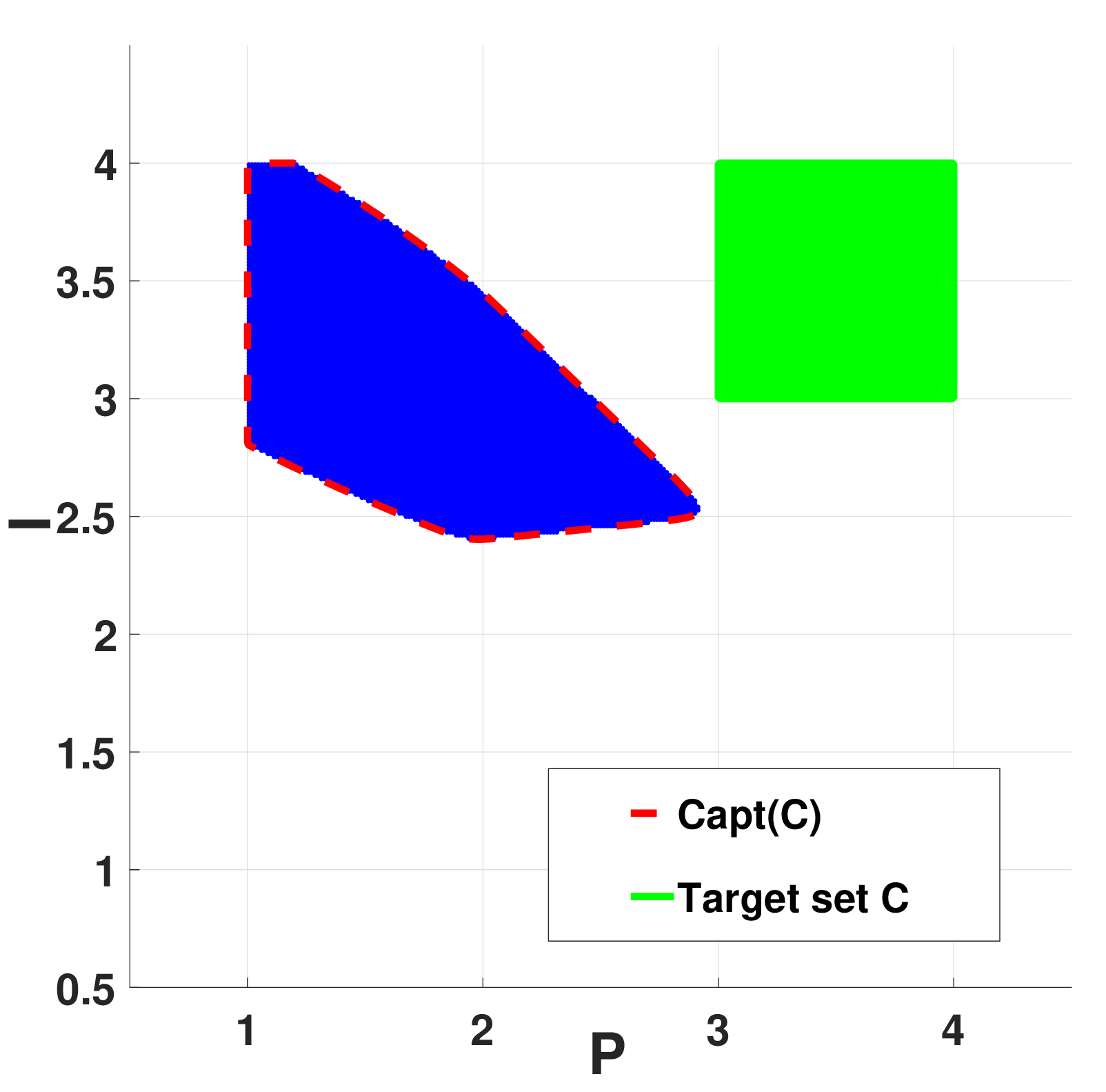}
	
	\caption{Capture basin for Viable Production and Inventory Management: Level sets of the value function $v$. The contour lines represent $v = 0$ (red) and $v < 0$ (blue). The target set $C$ is shown in green.}
	\label{capt_graph_2d}
\end{figure}

On fig.\ref{capt_graph_3d}, show the values of $v(t,.)$ at $t=0$ and $v(0, P, I)=0.5$, indicating an unreachable zone. On fig.\ref{capt_graph_2d}, iso-values are displayed.

The red contour in fig.\ref{capt_graph_2d} represents the capture basin $\operatorname{Cap}_{C}(0)$. It provides valuable insights into the dynamics of production $(P)$ and inventory $(I)$ underpricing and advertising policies, with 3D surfaces and contour lines depicting that the cost function $v$ varies with both variables. Furthermore, the capture basin delineates combinations of production and inventory where operational constraints remain acceptable, adhering to operational constraints; the green contour indicates target areas indicating ideal states where strategic objectives have been fully achieved.

Starting within the interior of the capture basin, companies can strategically adjust pricing and advertising policies to move closer to their target area. By doing this, companies can ensure that they remain within an acceptable performance range while progressing toward their strategic goals without breaching the operational constraints defined by $\hat{K}$ sets. Therefore, fig.\ref{capt_graph_2d} offers an excellent means of strategic decision-making, facilitating goal attainment while fulfilling the company's capacity constraints.
\begin{figure}[h!]
	\centering
	\includegraphics[width=0.8\textwidth]{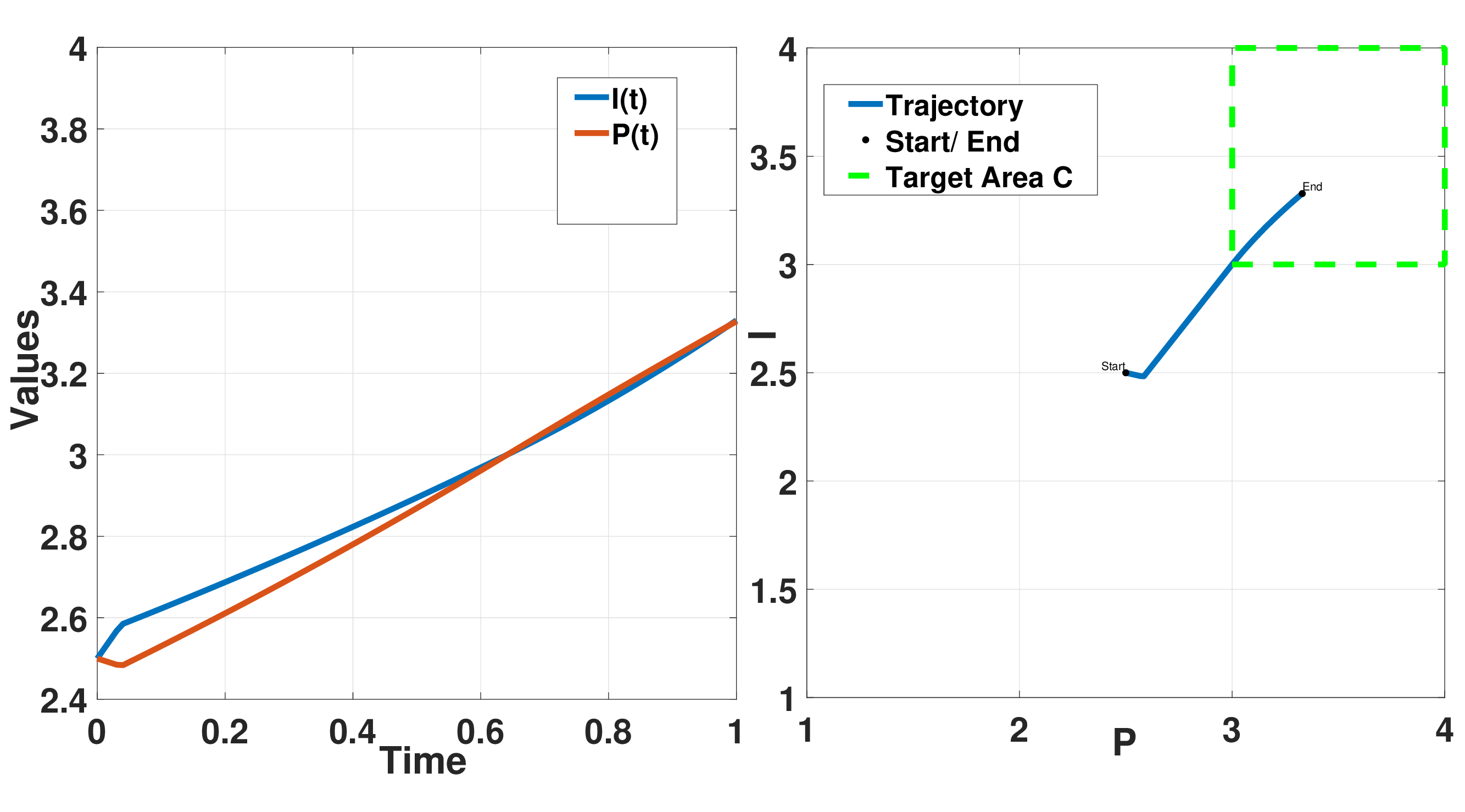}
	
	\caption{Optimal trajectory of \(P\) and \(I\) starting from \((P(0), I(0)) = (2.5, 2.5)\) and reaching the target set \(C\) at time \(T\).}
	\label{trajectory}
\end{figure}

\begin{figure}[h!]
	\centering
	\includegraphics[width=0.8\textwidth]{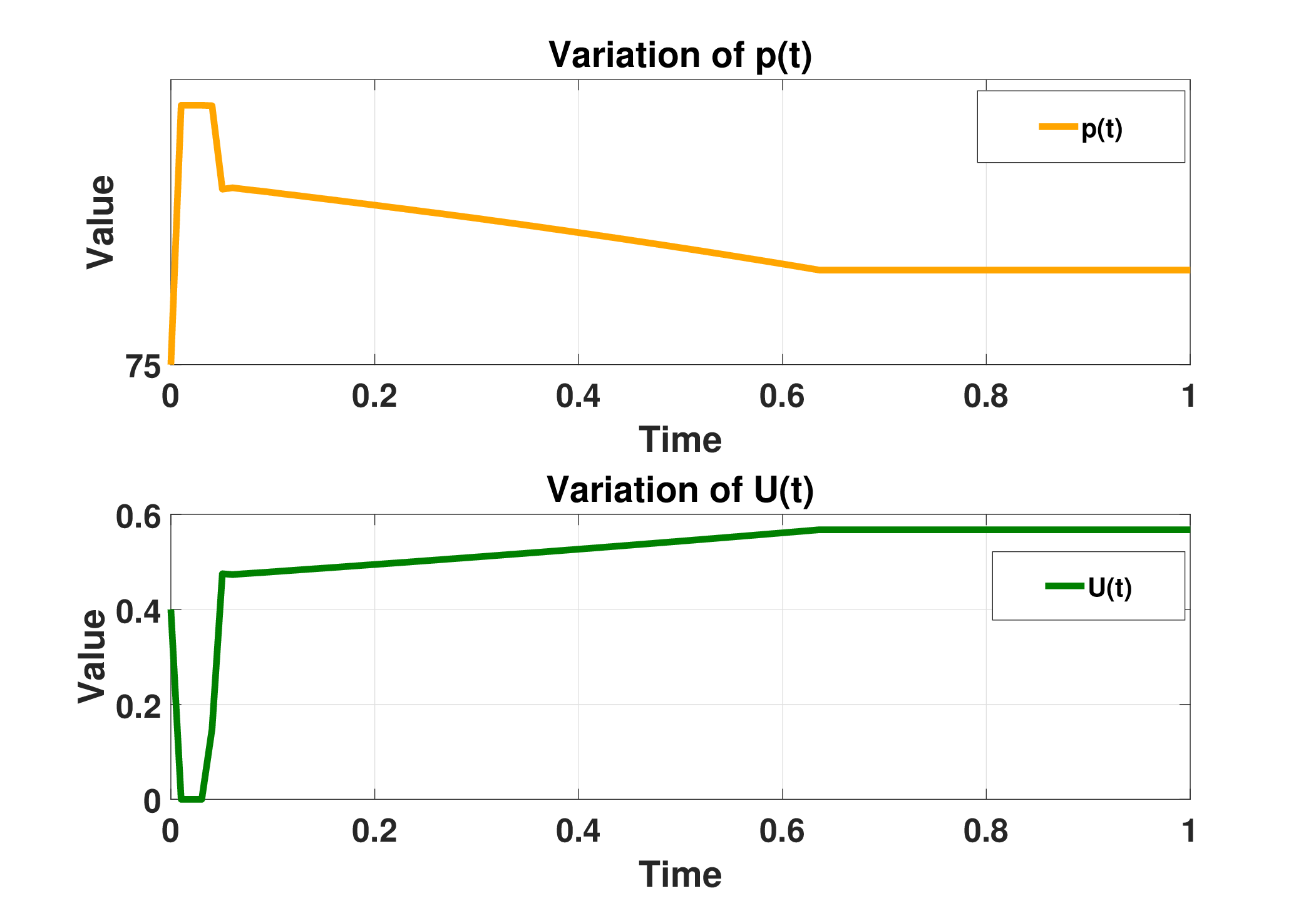}
	
	\caption{Variation of control parameters \(p(t)\) and \(U(t)\) in the feedback control system.}
	\label{control}
\end{figure}

Fig.\ref{trajectory} illustrates the trajectory of the controlled system's variables \( P \) and \( I \). The plot on the left shows the evolution of \( P \) concerning \( I \) during the simulation period. The black dots indicate the initial condition (\( P(0) = 2.5 \), \( I(0) = 2.5 \)) and the final state of the system. We chose the initial condition within the capture basin. This choice guarantees that the control strategies will effectively guide the system trajectory towards the target area. The green dashed line delineates the boundary of the target area, confirming that the trajectory is directed towards the final state within this target zone by the end time T.
In fig.\ref{control}, we have the variations of the control parameters \( p(t) \) and \( U(t) \) over time. showing the time-dependent changes in the control parameter \( p(t) \) and  \( U(t) \). These plots highlight how the control inputs were adjusted throughout the simulation to maintain the system in the target zone. By effectively tuning \( p(t) \) and \( U(t) \), the feedback control mechanism ensures that the system trajectory remains within the defined constraints and progresses towards the final target state. These variations illustrate the dynamic adjustments of the control system to achieve the desired objective.\\
After computing the capture basin for the short-term vision model, the next step was to examine how different mesh sizes affected the results. We compared the capture basins generated by various mesh resolutions using a \(1025^2\) mesh as the reference.
This comparison aimed to determine how the mesh size affects the accuracy and clarity of the capture basin shape. We can assess the trade-offs between computational complexity and precision by analyzing these differences in the results. This will help identify the most practical mesh size for simulations, balancing efficiency and details.\\
The inverse relationship between price \(p(t)\) and advertising effort \(U(t)\) observed in the control dynamics reflects an optimal balance between demand generation and resource allocation. When prices increase, advertising decreases because higher prices reduce consumer demand, making additional advertising less effective. This strategy conserves resources by prioritizing profitability over aggressive market expansion strategies. Conversely, when prices decrease, advertising increases to capitalize on lower prices and boost market penetration. This approach leverages the increased consumer responsiveness to lower prices, thereby maximizing the return on advertising investment. The alternation between high prices with low advertising and low prices with high advertising ensures that the system remains efficient and viable, maintaining demand within manageable levels and progressing toward the target state. This feedback mechanism aligns pricing and advertising strategies with market conditions and system constraints, achieving a dynamic balance that supports both short-term and long-term sustainability.

\begin{figure}[h!]
	\centering
	\includegraphics[width=0.49\textwidth]{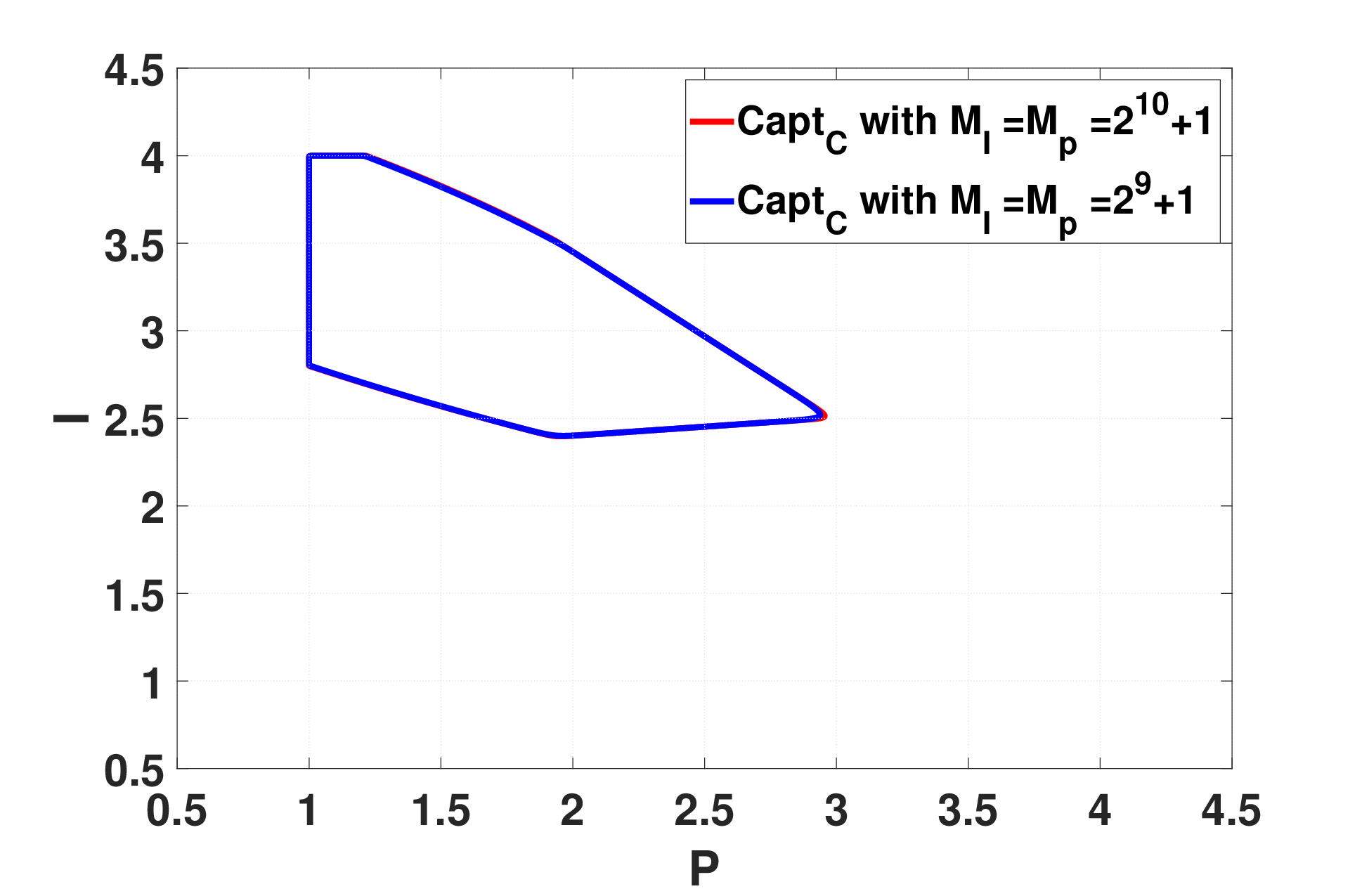}
     \includegraphics[width=0.49\textwidth]{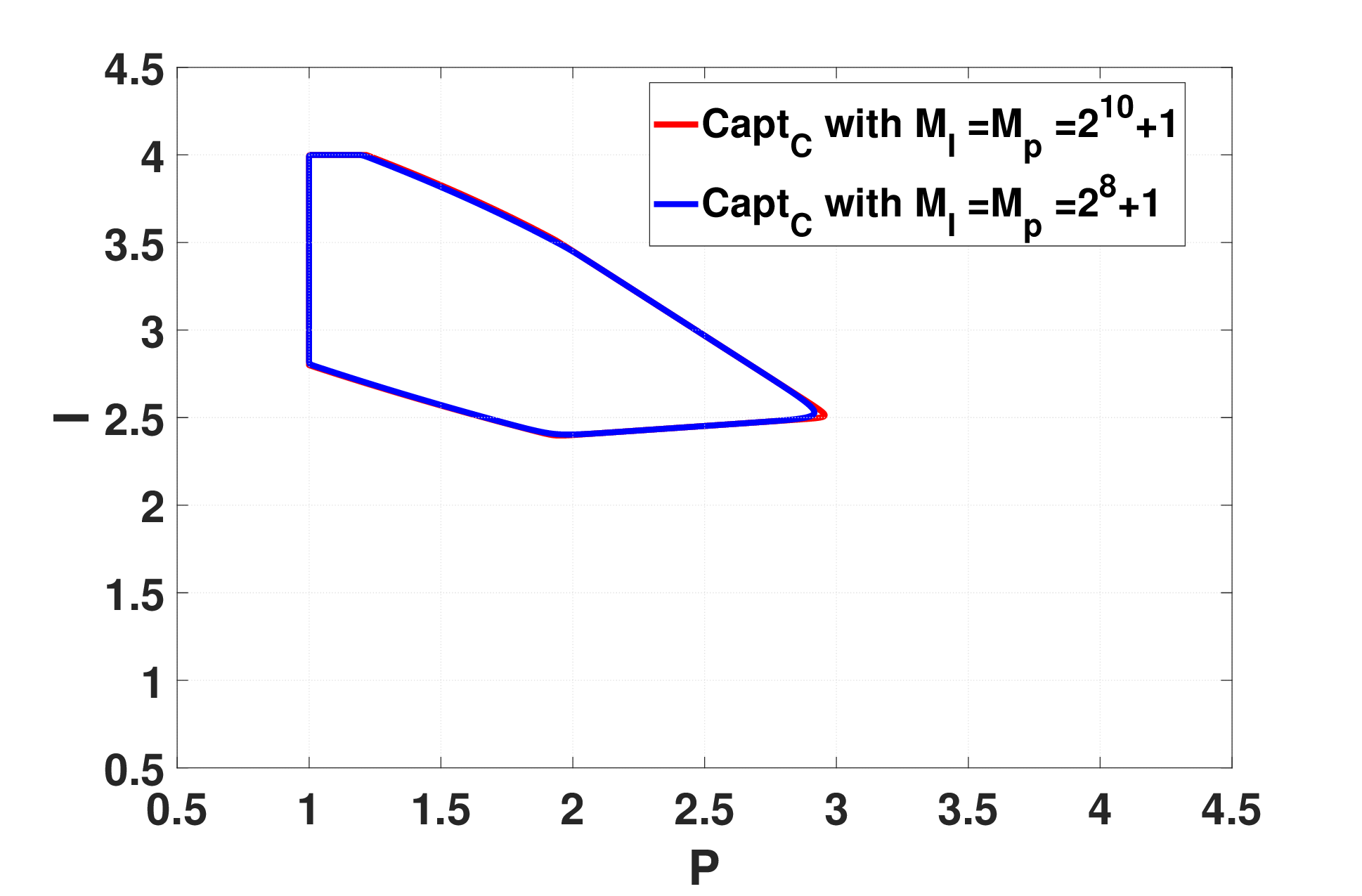}
       \includegraphics[width=0.49\textwidth]{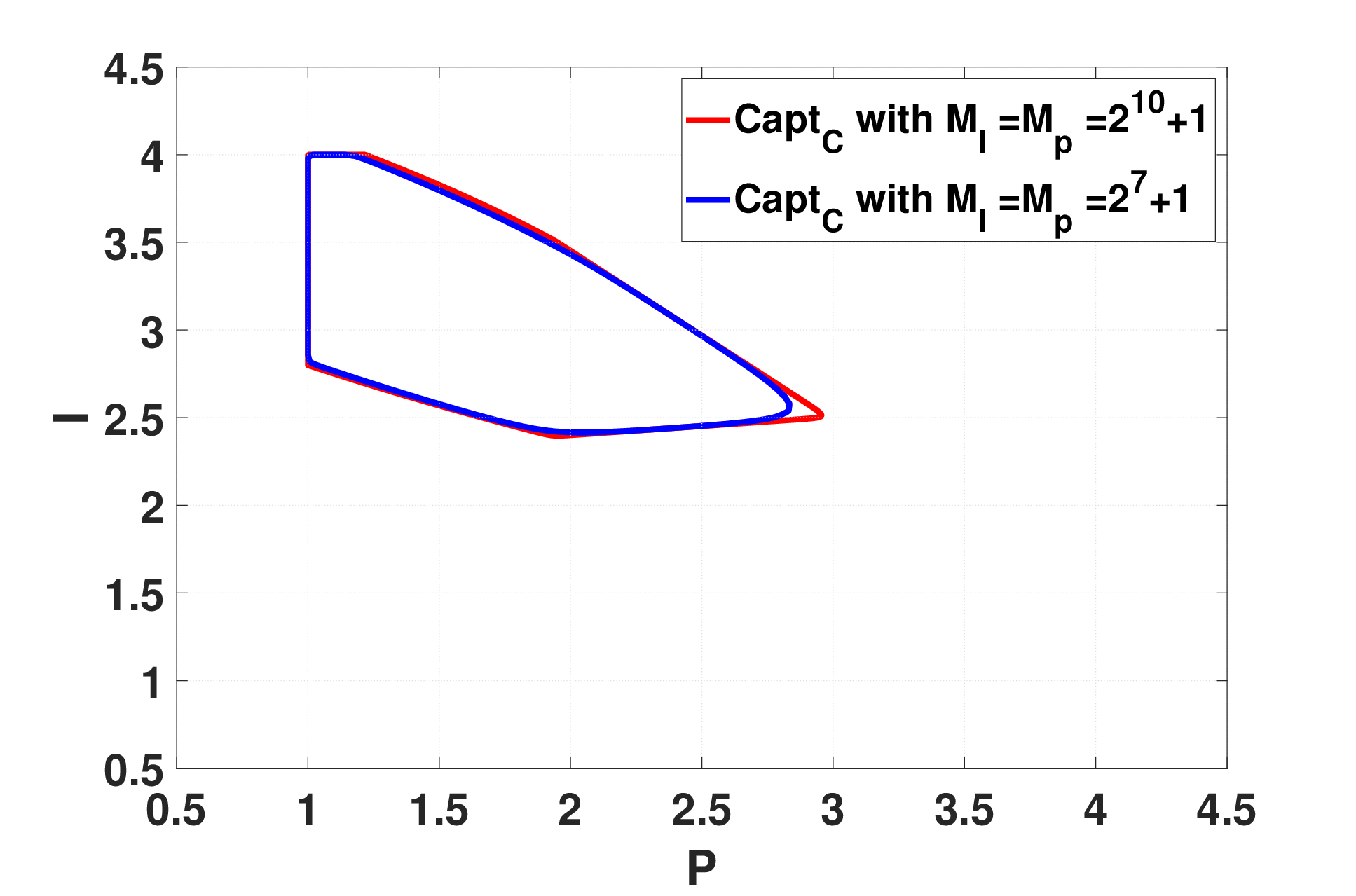}
        \includegraphics[width=0.49\textwidth]{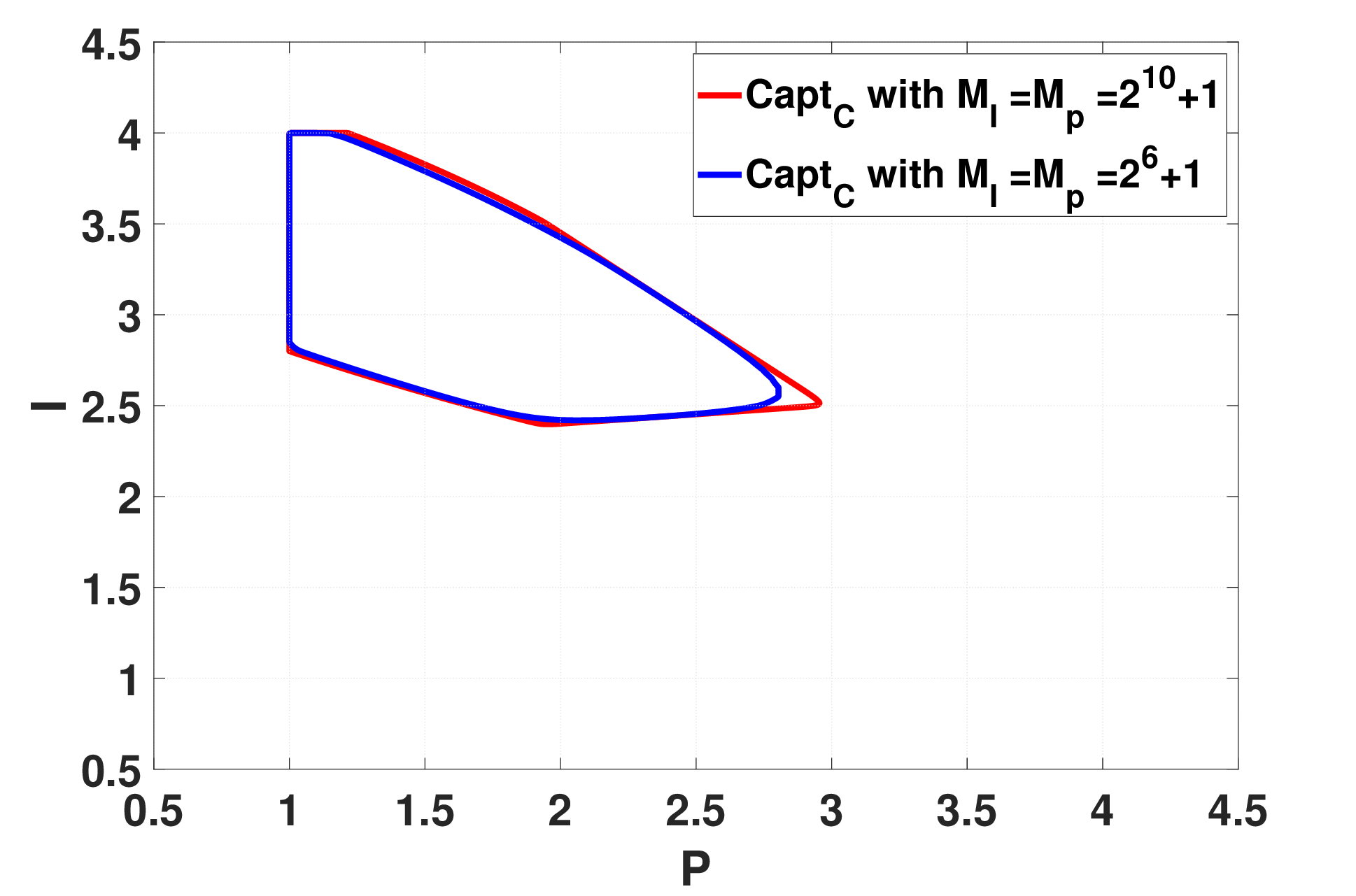}

	\caption{Numerical convergence of capture basin level as grid resolution increases. Each subfigure compares the level sets of the value function $v$ at various resolutions. }
	\label{comparaison}
\end{figure}
As shown in Figure \ref{comparaison}, a clear trend of numerical convergence is evident as the grid resolution increases. The capture basin computed at the finest resolution (\(M_I = M_P = 1025= 2^{10} +1\)) serves as the reference basin (\(v_0\)). The results indicate a consistent reduction in the \(L^\infty\) norm error as the grid resolution improved. Specifically, the error decreased from 0.90 at a coarse resolution (\(M_I = M_P = 2^6 +1\)) to 0.011 at a mid-range resolution (\(M_I = M_P = 2^9 +1\)). This decreasing trend underscores the effectiveness of finer meshes in approximating the reference solution, highlighting the importance of mesh resolution in achieving greater accuracy in the computational simulations. Table \ref{Tabmerch} provides a detailed numerical comparison, presenting the \( L^\infty \) norm between the numerical capture basin \(v_i\) for \(i \in \{1, 2, 3\}\) and the reference basin \(v_0\), highlighting the improvements in accuracy as the resolution approaches that of the reference grid.

\begin{table}[!h]
\centering
\caption{The \( L^\infty \) norm error between each computed front and the reference front \( v_{0} \), corresponding to \( M_I = M_P = 1025 \), using the same data as in Figure \ref{comparaison}.
} \label{Tabmerch}
\begin{tabular}{c c c c }
\hline
$M_{I}=M_{P}$ & $\Delta t$ & $CFL$ & $Error$ \\
\hline
513  & 0.0079  & 0.8 & 0.011  \\
257 & 0.00158 & 0.8 & 0.024 \\
129 & 0.00315 & 0.8 & 0.074 \\
65  & 0.0062  & 0.8 & 0.90 \\
\hline
\end{tabular}
\end{table}

\subsection{Long term numerical simulations}\label{longnumeric}

In the long-term numerical simulations, we incorporated the quality dimension \( q \), transforming the problem into three dimensions. This results in a state space defined by the parameters listed in Table \ref{Tab.2}, with the target set represented as described above. The main objective is not only to reach the target set \( C' \) but also to ensure that the system consistently remains within the viability kernel \( \operatorname{Viab}(C') \) over time. This approach highlights the importance of maintaining product quality within defined ranges to support high production and low inventory levels.

To achieve our objectives, we employed the Lax-Friedrichs (LF) scheme for the numerical approximation of the partial differential equation (PDE) model \eqref{HJB}. This discretization method provides a robust framework for effectively approximating the HJB equation, which is crucial for short-term control of the system. The specific configuration parameters used in \textit{Roc-hj} (\cite{RocHJ}) for this example are listed in Table \ref{Tab.2}.

\begin{table}[htbp]
\centering
\caption{Discretization scheme and {\tt Roc-hj} settings.}
\label{Tab.2}
\begin{tabular}{p{0.4\linewidth}c}
\hline
\textbf{Description} & \textbf{Value} \\
\hline
Discretization method & {\tt Lax-Friedrichs} \\
State space constraint \( K \) & \([1, 4] \times [20, 50] \times [1, 4]\) \\
Target set \( C' \) & \([1, 2] \times [25, 45] \times [3, 4]\) \\
Time parameter \( t_1 \) & \(1\) \\
\hline
\end{tabular}
\end{table}

The analysis was conducted in two phases. The first phase focuses on determining the viability kernel \( \operatorname{Viab}(C') \) to identify the control strategies that keep the system within the target set \( C' \). The second phase aims to target this viability kernel by computing the capture basin \( \operatorname{Cap}(\operatorname{Viab}(C')) \) at a specified time \( t_1 \), which indicates the time required to enter zone \( C' \).

Additional details on the numerical schemes used for the two PDEs are provided in Appendix \ref{app}.




\textbf{Phase 1: Viability Kernel Calculation:}\\

In the first phase, we focus on finding the viability kernel \(\operatorname{Viab}_{C'}\). This is achieved by solving the PDE in \eqref{HJB1}.
The solution to this PDE allows us to identify the initial states from which the system can be controlled to remain in the target zone \(C'\) for all future time. Figure \ref{kernel_3d} presents two cross-sectional views of the viability kernel within a 3D model, illustrating how the viability of the system changes along different dimensions. In the left panel, the cross-section is taken at \(I = 1.2\), emphasizing how the viability kernel behaves in the \(q\) dimension. This visual representation helps clarify the range of values for \(q\) in which the system remains viable, indicating the extent to which it can be controlled to remain within the target zone. In contrast, the right panel of Figure \ref{kernel_3d} provides a cross-section at \(q = 30\), showing the distribution of the viability kernel in the \(I\)-dimension. By examining both panels, it becomes evident how the viability of the system depends on changes in the \(I\) and \(q\) axes, offering insights into the combinations of these state variables that ensure that the system remains viable over time. These cross-sections provide a clear understanding of the dynamic regions in which the system can be maintained within the predefined constraints, contributing to the overall characterization of the viability kernel.
\begin{figure}[h!]
	\centering
	\includegraphics[width=0.49\textwidth]{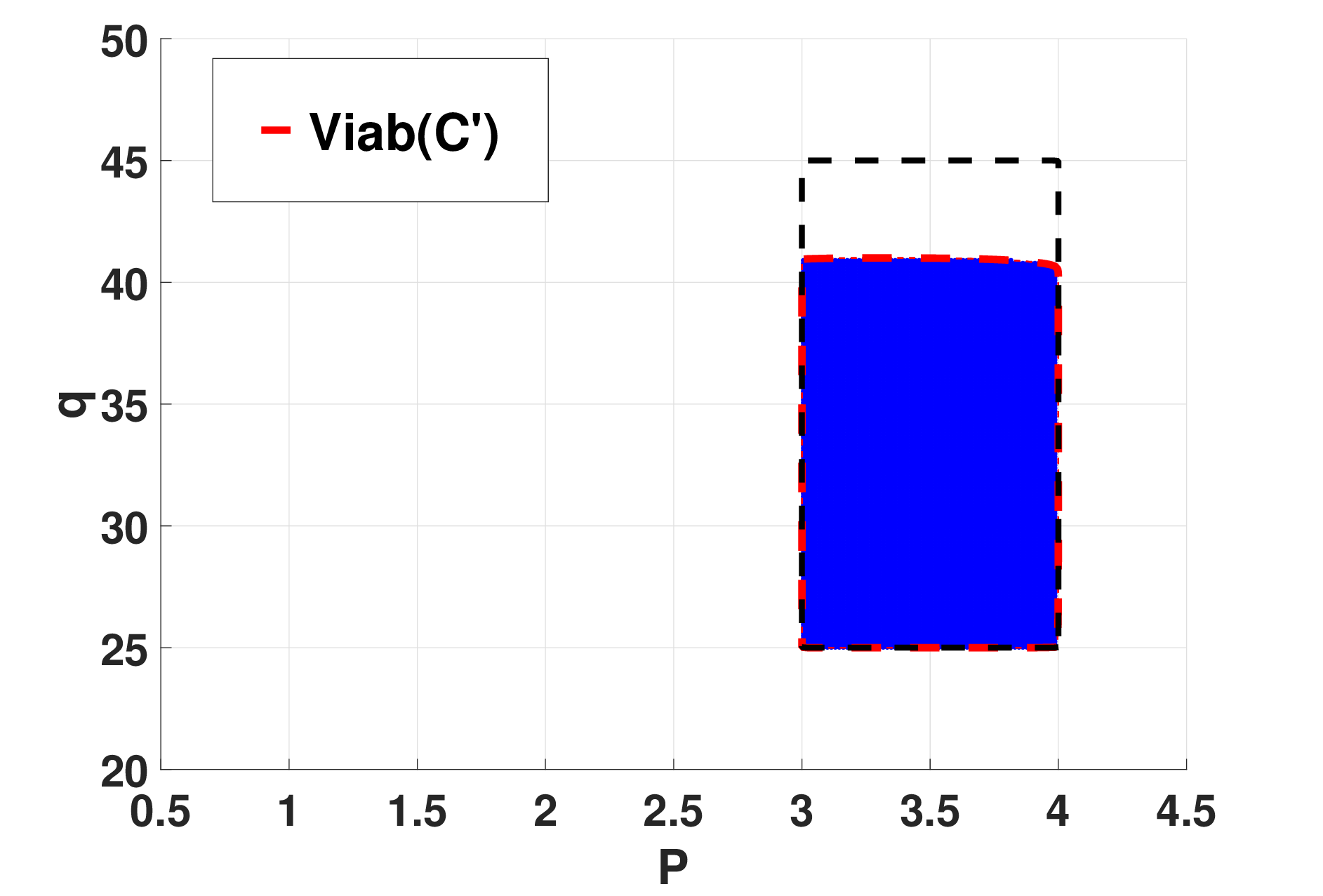}
     \includegraphics[width=0.49\textwidth]{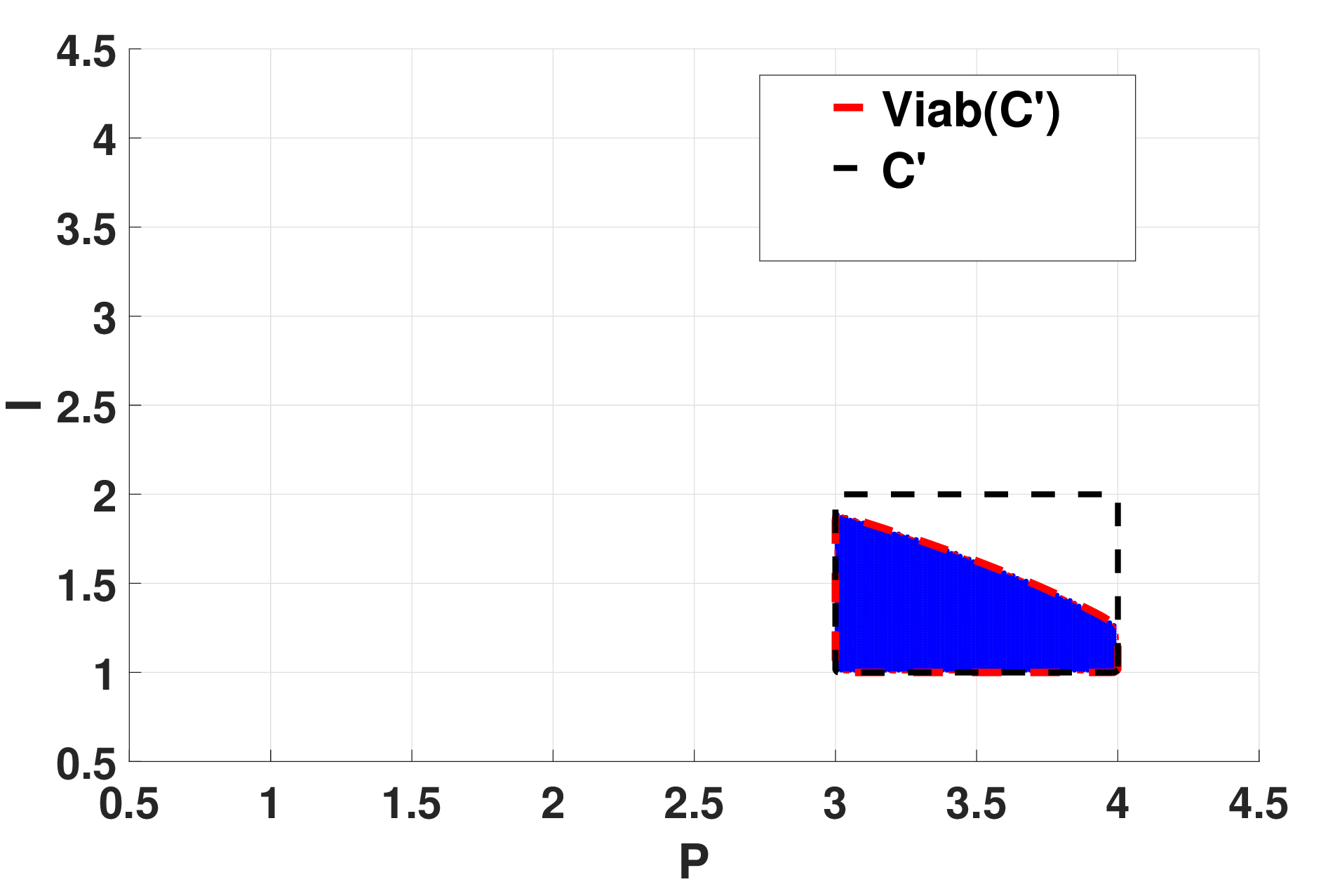}

	\caption{Viability kernel of the 3D model shown in different cross-sections. The left figure presents the kernel at \(q=30\), and the right figure displays the kernel at \(I=1.2\). These cross-sections help visualize the kernel's behavior across different axes.}
	\label{kernel_3d}
\end{figure}

\textbf{Phase 2: capture basin Determination:}\\
In the second phase, the focus shifts to determining the capture basin \(\operatorname{Cap}_{\operatorname{Viab}_{C'}}\), which consists of the set of initial conditions from which the system can be driven into the viability kernel by a specific time \(t_1 = 1\). This is accomplished by solving the PDE in \eqref{HJB2}, which helps us understand the control strategies required to ensure that the system reaches the viability kernel within the given entry time $t_1$. Figure \ref{capt_3d} shows the cross-sectional views of the capture basin for the 3D model, which is analogous to the representations of the viability kernel in the previous phase. In the left panel, the cross-section is taken at \(I = 1.2\), illustrating the behavior of the capture basin along the \(q\)-dimension. This visualization shows the extent to which the initial conditions in the \(q\) dimension can be controlled to reach the viability kernel by \(t_1 = 1\). The right panel shows a cross-section at \(q = 30\), highlighting how the capture basin changes across the \(I\)-dimension. Thus, the relationship between the initial system states and the required control effort to drive the system toward the viability kernel becomes clearer. Together, the cross-sections in Figure \ref{kernel_3d} allow for a detailed examination of the regions from which the system can be captured and directed into the viability kernel within the given time constraints, thus enhancing the understanding of the capture basin's structure and behavior in both the \(I\) and \(q\) dimensions.

\begin{figure}[h!]
	\centering
	\includegraphics[width=0.49\textwidth]{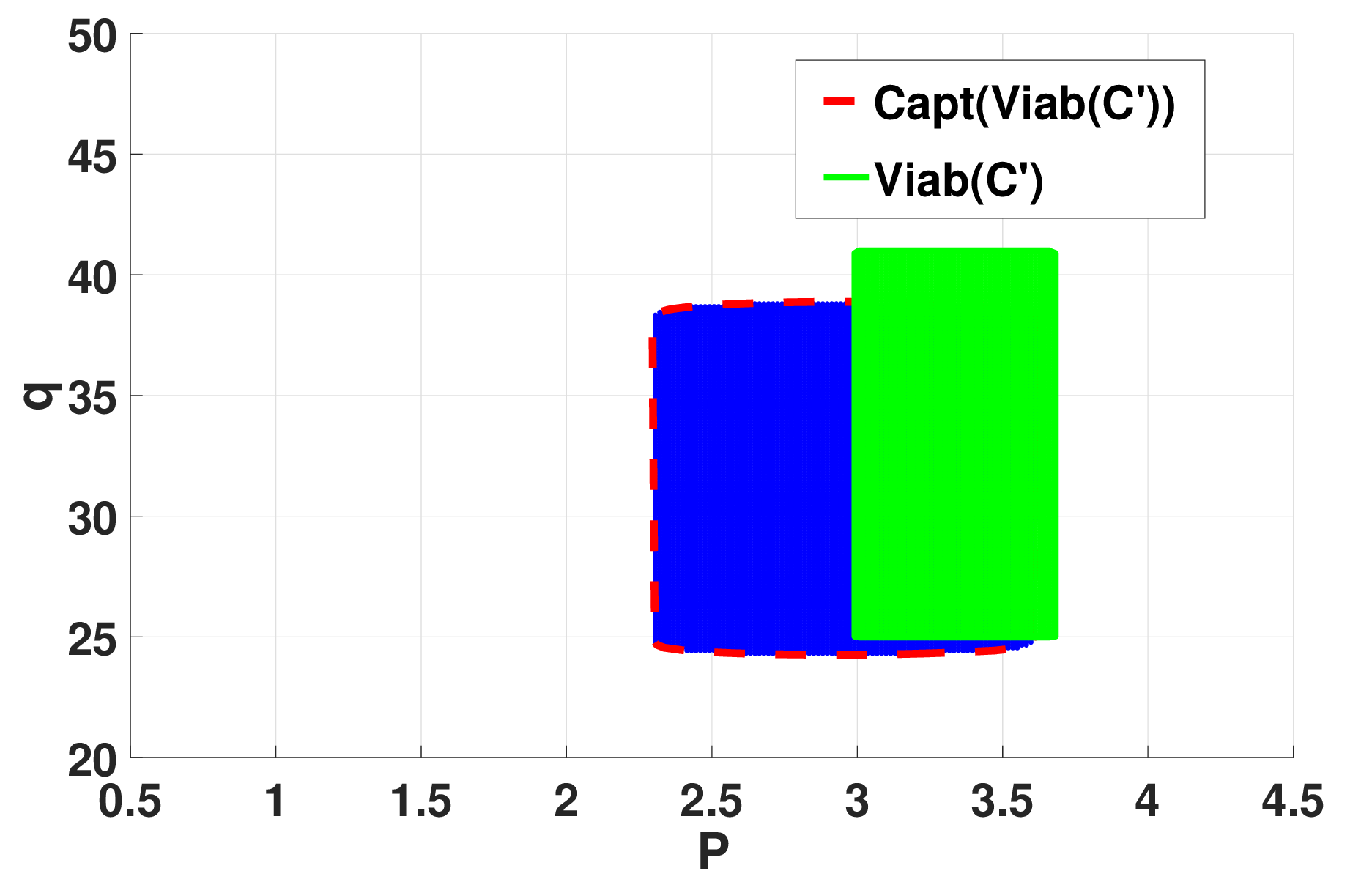}
     \includegraphics[width=0.49\textwidth]{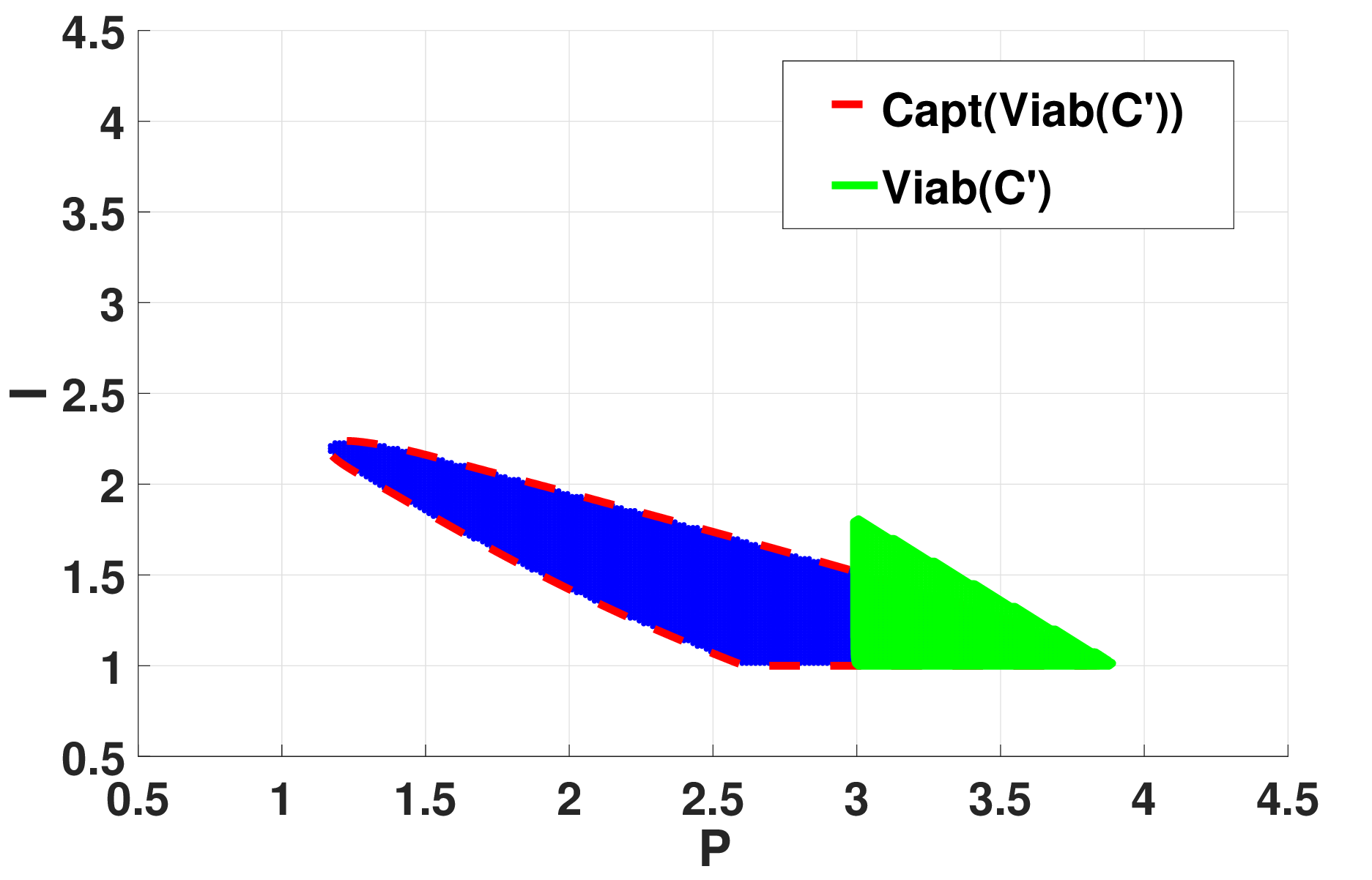}

	\caption{Capture basin of the 3D model shown in different cross-sections. The left figure presents the capture basin at \(q=30\), and the right figure displays the capture basin at \(I=1.2\). These cross-sections help visualize the capture basin's behavior across different axes.}
	\label{capt_3d}
\end{figure}
\newpage

To assess the robustness of our results, we conducted a sensitivity analysis to examine how variations in key parameters influenced the capture basin and viability kernel. The findings are presented in the following subsections.  
  
\subsection{Sensitivity analysis}\label{sensivity}

To enhance the robustness and practical applicability of our findings, we conducted a sensitivity analysis to investigate how variations in key parameters affect the capture basin and viability kernel, as computed in Subsections \ref{shortnumeric} and \ref{longnumeric} Specifically, we varied the demand function coefficients ($a, b, c, \rho$) and the quality decay rate ($\delta$), as these parameters significantly influence the system dynamics, including demand generation, pricing strategies, and quality evolution. The analysis was performed using the ROC-HJ solver with the same numerical setup as described in Subsections 6.1 and 6.2, ensuring consistency with the baseline results.
We selected realistic ranges for each parameter based on their physical and economic interpretations, as informed by studies such as \cite{sethi2008optimal} and \cite{anton2023dynamic}. The ranges were as follows:

\begin{table}[htb!]
\centering
\caption{Parameter ranges tested in the sensitivity analysis}
\label{sensitivity_params}
\begin{tabular}{l c c p{0.4\linewidth}}
\toprule
\textbf{Parameter} & \textbf{Baseline Value} & \textbf{Tested Range} & \textbf{Interpretation} \\
\midrule
\(a\) & 20 & [15, 25] & Baseline demand level \\
\(b\) & 0.2 & [0.15, 0.25] & Sensitivity of demand to price \\
\(c\) & 0.3 & [0.2, 0.4] & Impact of quality on demand \\
\(\rho\) & 0.2 & [0.15, 0.25] & Effect of advertising effort \\
\bottomrule
\end{tabular}
\end{table}

\begin{figure}[h!]
    \centering
    \subfloat[Sensitivity to parameter a\label{fig:sensivity_a}]{%
        \includegraphics[width=0.49\textwidth]{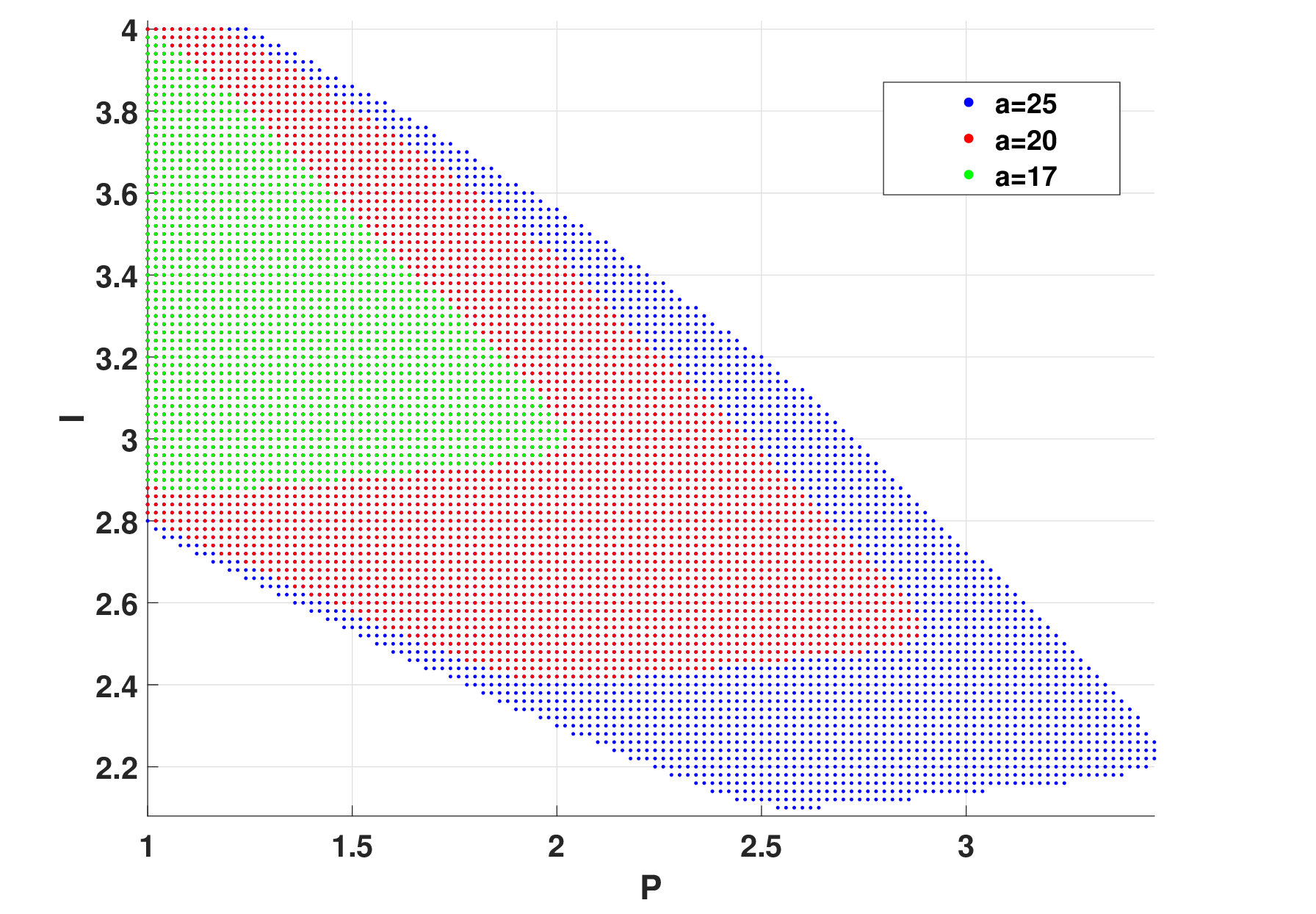}%
    }\hfill
    \subfloat[Sensitivity to Parameter b\label{fig:sensivity_b}]{%
        \includegraphics[width=0.49\textwidth]{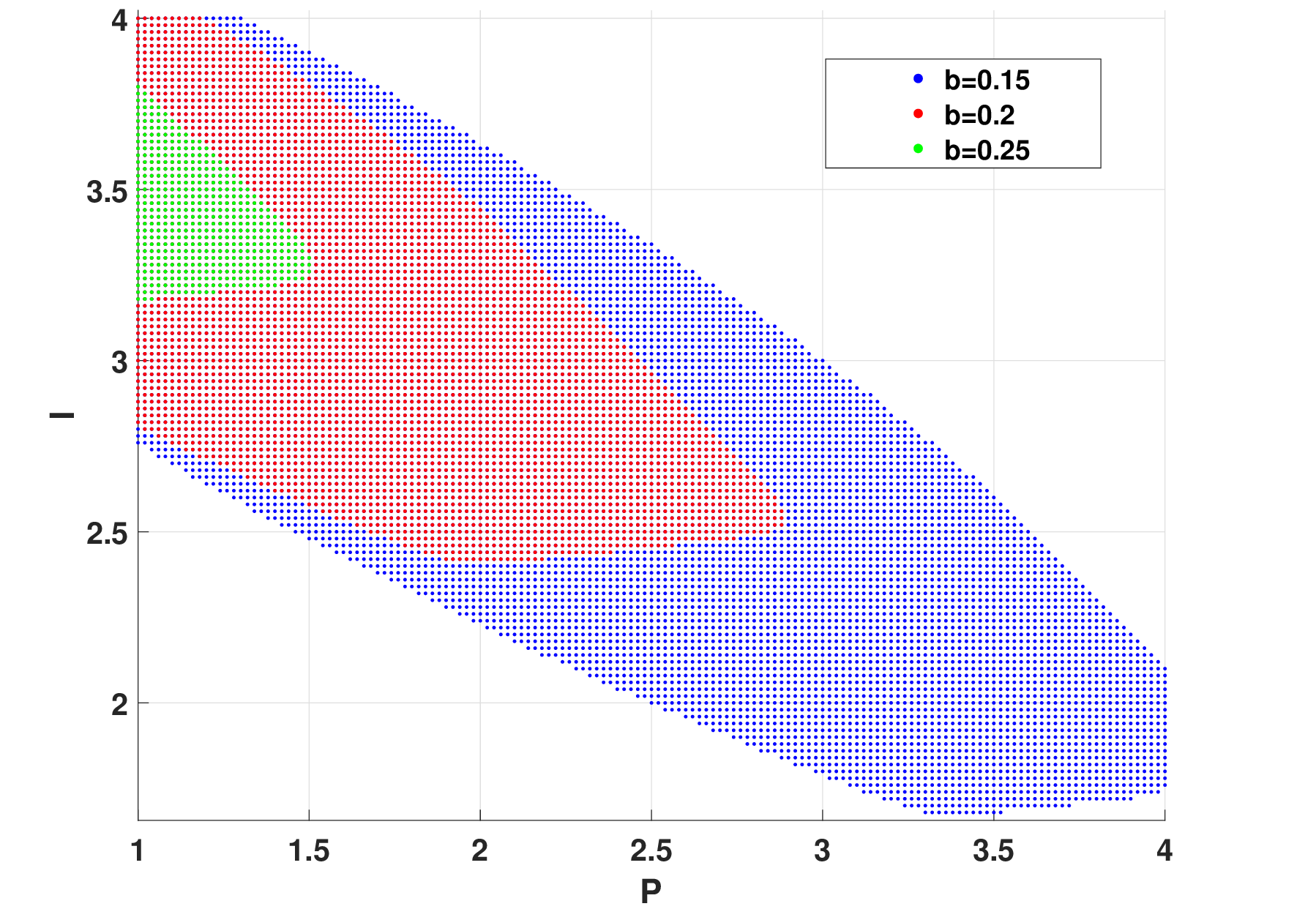}%
    }
    
    \subfloat[Sensitivity to Parameter c \label{fig:sensivity_c}]{%
        \includegraphics[width=0.49\textwidth]{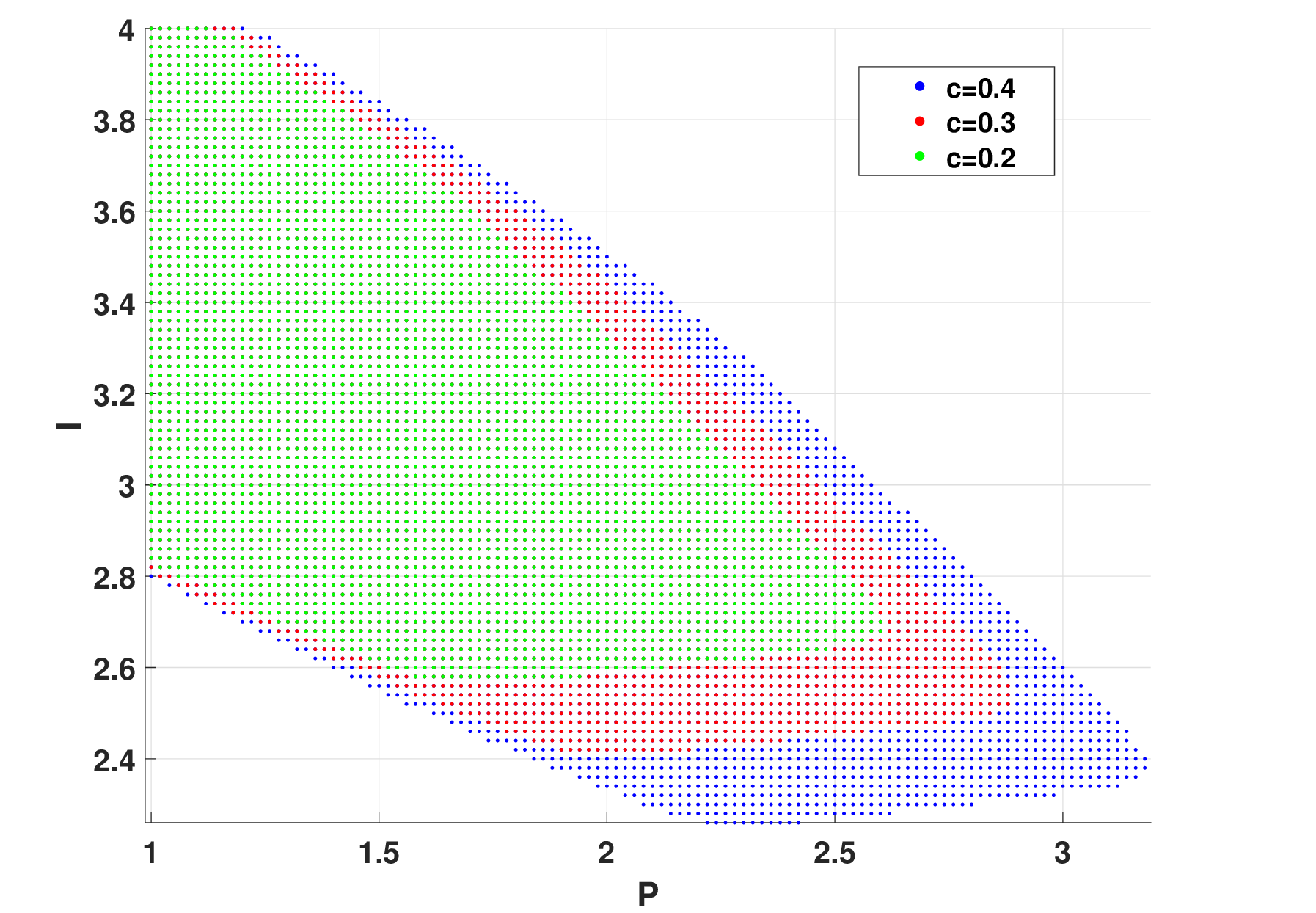}%
    }\hfill
    \subfloat[Sensitivity to Parameter $\rho$ \label{fig:sensivity_rho}]{%
        \includegraphics[width=0.49\textwidth]{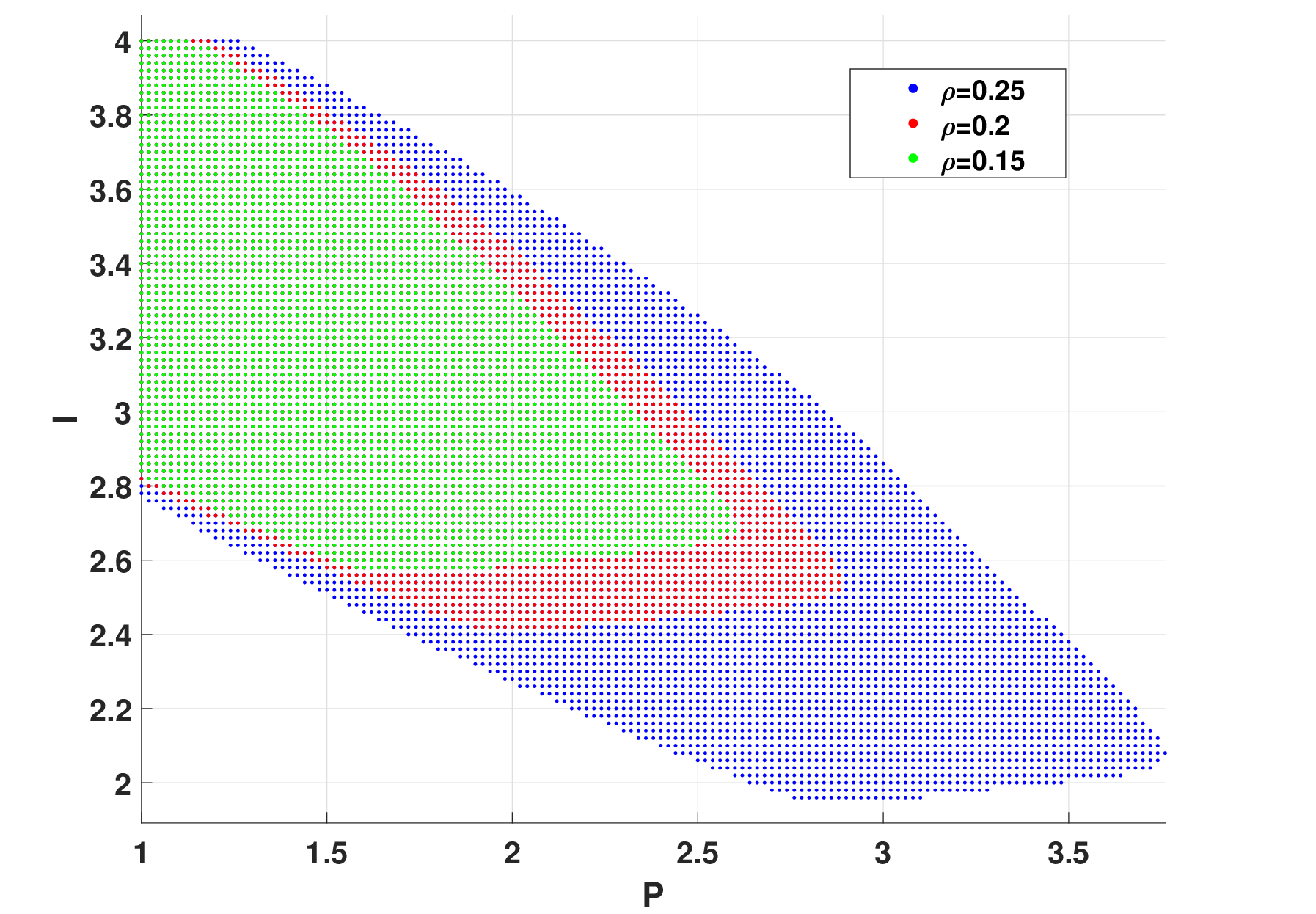}%
    }
    
    \caption{Sensitivity Analysis of Capture Basin Dynamics}
    \label{sensivity}
\end{figure}

Figure \ref{sensivity} illustrates the sensitivity of the capture basin $\operatorname{Cap}_{}(C)$ to variations in the demand function $\rho U(a - bp + cq)$, as established in the short-term numerical simulations (Subsection \ref{shortnumeric}). The analysis utilizes the ROC-HJ solver with parameter ranges and three baseline points from Table \ref{sensitivity_params}.\\
Subfigure (a) shows the effect of $ a $ (15 to 25), with the capture basin expanding as $ a $ increases from 15 (green) to 22 (blue), spanning $ P \approx 1 $ to 3 and $ I $ from 2.2 to 4 for $ a = 25 $, indicating that a higher baseline demand, modulated by $ U $, widens the controllable production-inventory range. Subfigure (b) highlights $ b $ (0.15 to 0.25), where the basin uniquely decreases as $ b $ rises, reflecting a stabilizing effect through the price sensitivity term $-bp$. Subfigure (c) examines $ c $ (0.2 to 0.4), with the basin expanding as $ c $ increases, suggesting an increased sensitivity to quality $ q $. Subfigure (d) assesses $ \rho $ (0.15 to 0.25), showing an expansion of the basin with higher $ \rho $, which enhances the influence of the advertising $ U $ on controllability.\\
These results emphasize that $ b $ is the sole parameter that reduces the capture basin with increasing values, whereas $ a $, $ c $, and $ \rho $ consistently expand it, aligning with the operational dynamics observed in the numerical simulations. This insight is crucial for optimizing production-inventory strategies under different market conditions.

\section{Conclusion}
This study established a comprehensive mathematical framework for viability analysis in production and inventory management systems. We developed a rigorous dynamic model that incorporates production rates, inventory levels, pricing strategies, and advertising effects within a unified analytical framework. The characterization of capture basins and viability kernels provides both fundamental theoretical insights and practical computational tools for maintaining operational viability under various constraints.

The numerical implementation demonstrates the effectiveness of the framework across multiple time horizons. The short-term analysis establishes adaptive control mechanisms for pricing and advertising strategies, whereas the long-term perspective develops sustainable approaches for the production-inventory balance. The computational results validate the theoretical constructs while demonstrating their practical utility in operational decision-making.

The Hamilton-Jacobi-Bellman formulation employed in this study offers distinct theoretical advantages, particularly in providing global viability solutions without restrictive controllability assumptions. Future research will pursue two principal directions: a rigorous comparative analysis with alternative optimization methodologies, including data-driven techniques (e.g., RL and DNNs) to establish quantitative computational advantages, and extensions addressing stochastic perturbations and time-dependent constraints to enhance real-world applicability. These developments will further clarify the advantages of the framework over model predictive control and data-driven techniques, while expanding its operational scope.

The mathematical foundations presented here create new research opportunities in both theoretical optimization and applied production management. By integrating analytical rigor with practical relevance, this study makes substantial contributions to mathematical control theory while providing actionable insights for production system optimization.

\section{Appendix}\label{app}

For a given mesh size \(\Delta x > 0\) and \(\Delta t > 0\), the grid is defined as follows: \[ G:= \{ I \Delta x \mid I \in \mathbb{Z}^d \}. \] Here, \( x_I = I \Delta x \) represents the discrete spatial locations, and \( t_n = n \Delta t \) represents the discrete time steps. The approximation of the solution \( v \) at the grid node \((x_I, t_n)\) is denoted as \( v^n_I \). 

\subsection{Numerical scheme finite case }

We use a numerical Hamiltonian \(\mathcal{H}^{LF} : \mathbb{R}^d \times \mathbb{R}^d \times \mathbb{R}^d \to \mathbb{R}\) in the Lax-Friedrichs (LF) scheme. This scheme is defined by the following equations:

 \begin{equation}\label{LF} \begin{cases} \min \left( \frac{-v_I^{n+1} + v_I^n}{\Delta t} + \mathcal{H}^{LF}(x_I, D^+v^{n+1}(x_I), D^-v^{n+1}(x_I)), \, v_I^{n} - g(x_I, t^{n}) \right) = 0, \\ v_I^0 = \tilde{u}_0(x_I), \end{cases} 
 \end{equation} 
 
 Here, \(\tilde{v}_T\) represents an approximation of the initial condition, \(v_T\). 

The discrete spatial gradients for \(v^{n+1}\) at point \(x_I\) in \(\mathbb{R}^d\) are given by: \[ D^+v^{n+1}(x_I) = (D_{x_1}^+v^n(x_I), \ldots, D_{x_d}^+v^n(x_I)), \] \[ D^-v^{n+1}(x_I) = (D_{x_1}^-v^n(x_I), \ldots, D_{x_d}^-v^n(x_I)), \] 

For any function \(\varphi\), the discrete gradients \(D^\pm_{x_i}\varphi(x_I)\) are defined as follows: \[ D^\pm_{x_1}\varphi(x_I) = \pm \frac{\varphi(x_{I \pm 1}) - \varphi(x_I)}{\Delta x}. \] 

From equation (\ref{LF}), we can deduce the explicite form: \[ v^{n}_I = \min \left( v^{n+1}_I - \Delta t \cdot \mathcal{H}^{LF}(x_I, D^+ v^{n+1}(x_I), D^- v^{n+1}(x_I)), \, g(x_I,t_n) \right). \] 

The LF Hamiltonian, \(\mathcal{H}^{LF}\), is defined as: \[ \mathcal{H}^{LF}(x, \psi^+, \psi^-) := H(x, \frac{\psi^+_1 + \psi^-_1}{2}, \frac{\psi^+_2 + \psi^-_2}{2}) - \frac{C_1(x)}{2} (\psi^+_1 - \psi^-_1) - \frac{C_2(x)}{2} (\psi^+_2 - \psi^-_2), \] where \( C_i(x) \) is chosen such that \[ \max_{\psi} \left| \frac{\partial H}{\partial \psi_i}(x, \psi) \right| \leq C_i(x). \] 

To ensure that the scheme remains monotonic, the Courant-Friedrichs-Lewy (CFL) condition must be satisfied: \[ \frac{\Delta t }{\Delta x} \max_x (C_1(x) + C_2(x)) \leq 1. \]

\subsection{Numerical Scheme for the Infinite Case}

To numerically solve Equation (\ref{HJB2}), the numerical shemas is based to find a stationary solution over time for the following auxiliary time-dependent problem:
\begin{equation}\label{auxiliary}
\min \left\{ 
\frac{\partial u(t, x)}{\partial t} + \lambda u(t, x) + H\left(x, \nabla u(t, x)\right), \, u(t, x) - \Gamma(x) 
\right\} = 0, \quad x \in \mathbb{R}^3, \, t > 0,
\end{equation}
subject to the initial condition 
\[
u(0, x) = u_0(x).
\]

This problem is intrinsically linked to the original HJB equation (\ref{HJB2}) because the stationary solution of the time-dependent equation \eqref{auxiliary} converges to the solution of the HJB equation as \( t \to \infty \). Notably, for fixed parameters \(\lambda\) and \(\Gamma\), Equation (\ref{HJB2}) possesses a unique viscosity solution. Consequently, the iterative numerical scheme is guaranteed to converge to this stationary solution, provided that the initial condition \(u_0 =h(x)\).

The spatial derivatives in Equation \eqref{auxiliary} are discretized on a rectangular grid \(\mathcal{G}_h = (\mathcal{G}_h^I, \mathcal{G}_h^q, \mathcal{G}_h^P)\), where the grid spacing in each coordinate direction \(i\) is denoted by \(h_i\) for \(i = 1, 2, 3\). The solution is approximated at discrete time steps \( t_n = n \Delta t \), where \(n \in \mathbb{N}\). The numerical approximation of the viscosity solution at the grid point \((I_i, q_j, P_k) \in \mathcal{G}_h\) and at time \( t_n \) is denoted by \( u_{i,j,k}^n \).

The numerical Hamiltonian \(\mathcal{H}_{i,j,k}^n\) is defined as:
\[
\mathcal{H}_{i,j,k}^n := \mathcal{H}\left( \partial_I u_{i,j,k}^n, \partial_q u_{i,j,k}^n, \partial_P u_{i,j,k}^n \right),
\]
where \(\partial_I u_{i,j,k}^n\), \(\partial_q u_{i,j,k}^n\), and \(\partial_P u_{i,j,k}^n\) represent the discrete spatial derivatives in the \(I\), \(q\), and \(P\) directions, respectively.

In our simulations, we employed the Local Lax-Friedrichs (LLF) scheme for the numerical Hamiltonian, ensuring that it was Lipschitz continuous and monotone flux consistent with \( \mathcal{H} \) at the mesh points. This approach is analogous to the method discussed for the finite case.

As the numerical iterations proceed, the solution \( u^n_{i,j,k} \) converges to a stationary state \( u_h \), thereby solving the HJB Equation \eqref{HJB}. The viability kernel \( \tilde{Viab_H} \) is then approximated as a set of points in the discretized space \( \mathcal{G}_h \) where the converged solution \( u_h \) is non-positive:
\[
\tilde{Viab_H} := \left\{ y \in \mathcal{G}_h \mid u_h(y) \leq 0 \right\}.
\]
Here, \( \tilde{Viab_H} \) serves as an approximation of the true viability kernel \( \text{Viab}(H) \) of the system.

This comprehensive numerical scheme ensures that the viability kernel is accurately captured, providing a robust solution to the original HJB problem.

\subsection*{Acknowledgements}

The authors would like to thank Hasnaa Zidani and Olivier Bokanowski for their insightful discussions on the numerical methods for HJB equations. A. Bouhmady also acknowledges that part of this work was carried out during the doctorate visit at the LMI laboratory (INSA Rouen Normandie) and benefited from financial support provided by the "Institut Français Maroc"

\end{document}